\documentclass[10pt]{amsart}

 \usepackage{hyperref}

\hypersetup{colorlinks=true, urlcolor=blue, citecolor=blue, linkcolor=blue}

\usepackage{amsmath,amsfonts, latexsym,graphicx, footmisc, amssymb}

\newcommand{\gecc}{{\operatorname{gecc}}}
\newcommand{\U}{{\mathcal U}}
\newcommand{\0}{{\mathbf 0}}
\newcommand{\C}{{\mathbb C}}
\newcommand{\Z}{{\mathbb Z}}

\newcommand{\Q}{{\mathbb Q}}
\newcommand{\N}{{\mathbb N}}
\newcommand{\cL}{{\mathbb L}}

\newcommand{\W}{{\mathcal W}}

\newcommand{\strat}{{\mathfrak S}}
\newcommand{\Proj}{{\mathbb P}}
\newcommand{\hyp}{{\mathbb H}}
\newcommand{\supp}{\operatorname{supp}}

\newcommand{\im}{\mathop{\rm im}\nolimits}
\newcommand{\rank}{\mathop{\rm rank}\nolimits}
\newcommand{\arrow}[1]{\stackrel{#1}{\longrightarrow}}
\newcommand{\Adot}{\mathbf A^\bullet}
\newcommand{\Bdot}{\mathbf B^\bullet}
\newcommand{\Cdot}{\mathbf C^\bullet}
\newcommand{\Idot}{\mathbf I^\bullet}
\newcommand{\Fdot}{\mathbf F^\bullet}

\newcommand{\Pdot}{\mathbf P^\bullet}
\newcommand{\Qdot}{\mathbf Q^\bullet}
\newcommand{\Udot}{\mathbf U^\bullet}
\newcommand{\Kdot}{\mathbf K^\bullet}

\newcommand{\mfl}{{\mathfrak l}}
\newcommand{\p}{{\mathbf p}}
\newcommand{\cc}{{\operatorname{CC}}}
\newcommand{\ms}{{\operatorname{SS}}}
\newcommand{\mfm}{{\mathfrak m}}
\newcommand{\vdual}{{\mathcal D}}
\newcommand{\call}{{\mathcal L}}
\newcommand{\imdf}{{{\operatorname{im}}d\tilde f}}

\newcommand{\lotimes}{\ {\overset{L}{\otimes}}\ }
\newcommand{\lboxtimes}{\ {\overset{L}{\boxtimes}\ }}
\newcommand{\piten}{{\pi_1^*\Adot\lotimes\pi_2^*\Bdot}}

\newtheorem{defn0}{Definition}[section]
\newtheorem{prop0}[defn0]{Proposition}
\newtheorem{conj0}[defn0]{Conjecture}
\newtheorem{thm0}[defn0]{Theorem}
\newtheorem{lem0}[defn0]{Lemma}
\newtheorem{corollary0}[defn0]{Corollary}
\newtheorem{example0}[defn0]{Example}
\newtheorem{remark0}[defn0]{Remark}
\newtheorem{question0}[defn0]{Question}

\newenvironment{defn}{\begin{defn0}}{\end{defn0}}
\newenvironment{prop}{\begin{prop0}}{\end{prop0}}

\newenvironment{thm}{\begin{thm0}}{\end{thm0}}

\newenvironment{cor}{\begin{corollary0}}{\end{corollary0}}
\newenvironment{exm}{\begin{example0}\rm}{\end{example0}}
\newenvironment{rem}{\begin{remark0}\rm}{\end{remark0}}

\newcommand{\defref}[1]{Definition~\ref{#1}}
\newcommand{\propref}[1]{Proposition~\ref{#1}}
\newcommand{\thmref}[1]{Theorem~\ref{#1}}

\newcommand{\corref}[1]{Corollary~\ref{#1}}
\newcommand{\exref}[1]{Example~\ref{#1}}
\newcommand{\secref}[1]{Section~\ref{#1}}
\newcommand{\remref}[1]{Remark~\ref{#1}}

\newcommand{\equref}[1]{Formula~\ref{#1}}

\title[Calculations with Characteristic Cycles]{Calculations with Characteristic Cycles}

\subjclass[2010]{32B15, 32C35, 32C18, 32B10}
\keywords{characteristic cycle, constructible complexes, nearby cycles, vanishing cycles}

\author{David B. Massey}

\date{}

\begin{document}

\begin{abstract} There are many constructible complexes of sheaves of $\Z$-modules which arise in the study of the topology of complex hypersurfaces. But these complexes are difficult to calculate in any effective manner. We focus instead on a form of characteristic cycles for complexes of sheaves: a graded, enriched characteristic cycle, in which we replace integer coefficients with Morse modules of coefficients. This allows us to preserve much more data than a standard characteristic cycle, while at the same time allowing us to calculate in algebraic/geometric terms. 
\end{abstract}

\maketitle




\section{Introduction} 

Throughout this paper, we must assume that the reader is familiar with basic aspects of the derived category of bounded constructible complexes of sheaves, perverse sheaves, and the nearby and vanishing cycles. Good references for the theory are \cite{kashsch}, \cite{dimcasheaves}, and \cite{schurbook}.

\medskip

Suppose that $\U$ is a connected open subset of the origin in $\C^{n+1}$ and, for convenience, assume that $\0\in\U$.  Let $f:(\U, \0))\rightarrow(\C,0)$ be a complex analytic function, which is not identically zero, and consider the hypersurface $V(f)=f^{-1}(0)$. 

There are several spaces which are typically studied when investigating the local, embedded topology of $V(f)$ at $\0$: the complement $V(f)$ in a small ball around $\0$, the real link of $V(f)$ at $\0$, the complex link of $V(f)$ at $\0$, and the Milnor fiber of $f$ at $\0$.

If $\0$ is a non-isolated critical point of $f$, then it is important to look at each of these spaces at every critical point, and to know how the topology of $V(f)$ at these nearby critical points is related to the topology at $\0$. This is a question of how local data patches together to give global data. Thus, complexes of sheaves of $\Z$-modules naturally enter the picture. And so, it is beneficial to look at the constant sheaf on the complement of $V(f)$, the Verdier dual of the constant sheaf on $V(f)$, the vanishing cycles of the constant sheaf on $V(f)$ along a generic linear form, and the nearby and vanishing cycles of the constant sheaf on $\U$ along $f$.

But complexes of sheaves contain so much data that they are are not amenable to computation. Hence, researchers consider other objects, which do not contain as much data, but which can be calculated algebraically/geometrically. The characteristic cycle and micro-support of a complex of sheaves are prominent examples of two such objects. We will define these carefully later, but we wish to describe them briefly here.

Let $X$ be a complex analytic space inside $\U$. Let $\strat$ be a complex analytic Whitney stratification of $X$, with connected strata. We let $\overline{T^*_S\U}$ denote the closure of the conormal variety of $S$ in $\U$, that is, the closure of the set of $(\mathbf p, \eta)\in T^*\U$ such that $\mathbf p\in S$ and $\eta(T_{\mathbf p}S)\equiv 0$.

Now, let $\Fdot$ be a bounded complex of sheaves of $\Z$-modules, which is constructible with respect to $\strat$. Then, as described by Goresky and MacPherson \cite{stratmorse}, to each stratum $S$ in $\strat$, there are an associated normal slice $\N_S$ and complex link $\mathbb L_S$. The isomorphism-types of the hypercohomology modules $\hyp^*(\N_S, \mathbb L_S; \Fdot)$ are independent of the choices made in defining the normal slice and complex link; these are the {\it Morse modules} of $S$, with respect to $\Fdot$. We let $d_S:=\dim S$, and $m_S^k(\Fdot):=\hyp^{k-d_S}(\N_S, \cL_S; \Fdot)$.

The Morse modules $m_S^k(\Fdot)$ tell one how the cohomology of $\Fdot$ changes as one moves through the stratum $S$. We say that a stratum $S$ is {\it $\Fdot$-visible} provided that $m^*_S(\Fdot)\neq 0$, i.e., provided that there exists $k$ such that $m^k_S(\Fdot)\neq 0$. The $\Fdot$-invisible strata, the strata which are {\bf not} $\Fdot$-visible, are, in a sense, strata that could be omitted from the stratification as far as the cohomology of $\Fdot$ is concerned.

The union of the closures of conormal varieties to $\Fdot$-visible strata is the {\it microsupport}, $\ms(\Fdot)$, of $\Fdot$, as defined by Kashiwara and Schapira in \cite{kashsch}, i.e.,   
$$\ms(\Fdot):= \bigcup_{m^*_S(\Fdot)\neq 0}\overline{T^*_S\U}.
$$
The microsupport is microlocal data which encodes the directions in which $\Fdot$ changes at each point.

Let $c_S(\Fdot)$ be the Euler characteristic of the Morse modules, i.e., let  $c_S(\Fdot):=\sum_{k\in\Z} (-1)^{k}\rank(m^k_S(\Fdot))$. The {\it characteristic cycle of $\Fdot$} is the cycle 
$$
\cc(\Fdot)=\sum_{S\in\strat}c_S(\Fdot)\left[\overline{T^*_S\U}\right],
$$
where the square brackets indicate that we are considering the conormal varieties as cycles. Both the microsupport and the characteristic cycle are independent of the stratification $\strat$, and so are intrinsic to the complex $\Fdot$ (and the ambient space $\U$).

\medskip

But the microsupport and the characteristic cycle throw away a large amount of the Morse module data. Why not take the conormal varieties and, instead of weighting them with Euler characteristic data from the Morse modules, consider a formal sum in which the coefficients are the Morse modules themselves? 

Thus, we define $\gecc^\bullet(\Fdot)$, the {\it graded, enriched characteristic cycle of $\Fdot$}, by defined a ``cycle'' in degree $k$ as a formal sum of modules times conormal varieties:
$$
\gecc^k(\Fdot) :=\sum_{S\in\strat} m_S^k(\Fdot)\big[\,\overline{T^*_{{}_{S}}\U}\,\big] =  \sum_{S\in\strat} H^{k-d_S}(\mathbb N_S,\mathbb L_S; \Fdot)\big[\,\overline{T^*_{{}_{S}}\U}\,\big],
$$
and we use an intersection theory which is a very mild extension of the theory of proper intersection of cycles ({\bf not} cycle classes) inside a complex manifold, as described in \cite{fulton}.

\medskip

There have been numerous other works on the computations of characteristic cycles: notably, the papers of Ginsburg \cite{ginsburg}, Briancon, Maisonobe and Merle \cite{bmm}, and Parusi\'nski and Pragacz \cite{paruprag}, plus portions of the books of Kashiwara and Schapira \cite{kashsch} and of Sch\"urmann \cite{schurbook}. However, there are several advantages to the techniques and results presented here.

\smallskip

\begin{itemize}

\item The intersection theory that we use is that of properly intersecting cycles inside a complex manifold. For such intersections, there are well-defined intersections cycles, not cycle {\bf classes} (see \cite{fulton}, Chapter 8).

The fact that we have intersection cycles with fixed underlying analytic sets makes calculations and formulas much easier and algorithmic, and, typically, the amount of genericity that we need in statements is merely that the intersections are proper, which is a relatively simple thing to check.

It is an interesting aspect of the theory that, using only enough genericity to obtain proper intersections does not yield objects which are as generic as possible, and it is precisely this lack of really generic genericity that makes formulas work so well.

\smallskip

\item While we use an easy intersection theory, we use modules in various degrees for the coefficients of our cycles. This graded, enriched intersection theory adds essentially no difficulty to computations, and yet, almost magically, yields results on the levels of modules, instead of merely giving numerical results.

\smallskip

\item In addition to the notion of graded, enriched characteristic cycles, our primary new device involved in the calculus of $\gecc$'s is the {\it graded, enriched relative polar curve} (see \cite{enrichpolar} and \secref{sec:relpolarcurve} of this paper). This is a substantial generalization the now-classic relative polar curve introduced by Hamm, L\^e, and Teissier in 1973 in  \cite{hammlezariski}, \cite{teissiercargese}, \cite{leattach}, and \cite{letopuse}.

By giving the ``correct'' definition of the general polar curve, we are not required to make choices as generically as did Hamm, L\^e, and Teissier and, thus, once again, the genericity hypotheses that we need in theorems are simply that certain intersections are proper.

\smallskip

\item Our calculation of the $\gecc$ of the vanishing cycles uses a generalization of the L\^e cycle algorithm that we developed in \cite{lecycles}, and so really does allow for explicit calculations in many examples.

\end{itemize}

\medskip

\noindent Aside from using graded enriched cycles and a more general relative polar curve, what is new in this paper?

\medskip

\begin{itemize}

\item In Section 3,  we define a generalization of the relative local Euler obstruction, $\operatorname{Eu}_{\mathbf p}f$, as was introduced in \cite{bmps}. We then prove a number of fundamental properties which hold for $\operatorname{Eu}_{\mathbf p}f$.

\medskip

\item In \thmref{thm:psigecc} and \thmref{thm:vanpsi} of Section 6, we recall our previous formulas involving $\gecc^\bullet(\psi_f[-1]\Fdot)$. However, the example that we calculate is new, as is the easy -- but interesting -- corollary, \corref{cor:numnearnum}, that the shifted nearby cycles of numerical complexes (see \defref{def:numerical}) are numerical.

\medskip

\item In Section 7, \thmref{thm:vanrestrext} is an enriched cycle version of one of our earlier results, which relates to the calculation of the graded enriched characteristic cycle of the complement of a hypersurface. However, \corref{cor:morserestrext}, \corref{cor:restrextgeccset}, \thmref{thm:vanrestrpush}, and \corref{cor:numcompnum} are new. Furthermore, as we show in \exref{exm:genlink}, \corref{cor:hyplink} is a  new generalization of the classically-known formula for the number of spheres in the homotopy-type of the complex link of an affine hypersurface.

\medskip

\item In Section 8, we recall our earlier result, \thmref{thm:vangecc}, and then show, in \corref{cor:vannum}, that the shifted vanishing cycles of numerical complexes are numerical. \thmref{thm:sskphi} is new, but follows quickly from some of our previous results. In \remref{rem:lecyclerem}, we discuss a general method for calculating the graded enriched characteristic of the vanishing cycles; this uses a cycle approach which is a generalization of our method for calculating L\^e cycles (see \cite{lecycles}). In \exref{exm:vangecc}, we give an example of how the method of \remref{rem:lecyclerem} actually works in practice.

\end{itemize}

\bigskip

We reiterate that, throughout this paper, it is important that, when we state that a choice must be made ``generically'', we actually give effective means of checking that the choice is generic enough. This makes the results much more useful when applying them to specific examples, and we give sample calculations to illustrate this point.

\section[The Characteristic Cycle]{The Characteristic Cycle and the Graded, Enriched Characteristic Cycle}\label{sec:gecc}

Throughout this paper, we fix a base ring
$R$ that is a regular, Noetherian ring with finite Krull dimension (e.g., $\Z$,
$\Q$, or  $\C$). This implies that every finitely-generated $R$-module has finite projective dimension (in fact, it
implies that the projective dimension of the module is at most $\dim R$). 

We let $\U$ be an open neighborhood of the origin of $\C^{n+1}$, and let $X$ be a closed, analytic subset of $\U$. We let $\mathbf z:=(z_0, \dots, z_n)$ be coordinates on $\U$. Having fixed the coordinates, we identify the cotangent space $T^*\U$ with $\U\times\C^{n+1}$ by mapping $(\p, w_0d_{\p}z_0+\dots+w_nd_{\p}z_n)$ to $(\p, (w_0,\dots, w_n))$. Let $\pi:T^*\U\rightarrow\U$ denote the projection.  

Let $\strat$ be a complex analytic Whitney stratification of $X$, with connected strata. Let $\Fdot$ be a bounded complex of sheaves of $R$-modules on $X$, which is constructible with respect to $\strat$. For each $S\in\strat$, we let $d_S:=\dim S$, and let $(\N_S, \cL_S)$ denote {\it complex Morse data for $S$ in $X$}, consisting of a normal slice and complex link of $S$ in $X$; see, for instance, \cite{stratmorse} or \cite{numcontrol}.

\smallskip

A general reference for the remainder of this section is \cite{singenrich}. 

\begin{defn} For each $S\in\strat$ and each integer $k$, the isomorphism-type of the module
$m_S^k(\Fdot):=\hyp^{k-d_S}(\N_S, \cL_S; \Fdot)$ is independent of the choice of $(\N_S, \cL_S)$; we
refer to
$m_S^k(\Fdot)$ as the {\it degree $k$ Morse module of $S$ with respect to
$\Fdot$}.
\end{defn}

\smallskip

\begin{rem}\label{rem:morsemod} The shift by $d_S$ above is present so that perverse sheaves can have non-zero Morse modules in only degree $0$. 

We also remark that, up to isomorphism, $m_S^k(\Fdot)$ can be obtained in terms of vanishing cycles. To accomplish this, select any point $\p\in S$. Consider an analytic function $\tilde g:(\U^\prime, \p)\rightarrow (\C,0)$  on some open neighborhood of $\p$ in $\U$ such that $d_\p \tilde g$ is a nondegenerate covector (in the sense of \cite{stratmorse}), and such that $\p$ is a (complex) nondegenerate critical point of $\tilde g_{|_{\U^\prime\cap S}}$. Let $g:=\tilde g_{|_{\U^\prime\cap X}}$. Then, $m_S^k(\Fdot)$ is isomorphic to the stalk cohomology $H^k(\phi_g[-1]\Fdot)_{\p}$. 

 Note that, if $\0$ is a point-stratum, then $m^k_\0(\Fdot)\cong H^k(\phi_{\call}[-1]\Fdot)_\0$, where $\call$ is the restriction to $X$ of a generic linear form $\tilde\call$.

\end{rem}

\medskip

For any analytic submanifold $M\subseteq\U$, we denote the conormal space $$\{(\p, \omega)\in T^*\U\ |\ \omega(T_{\p}M)\equiv 0\}$$ by $T^*_M\U$, and will typically be interested in its closure $\overline{T^*_M\U}$ in $T^*\U$.

\begin{defn}
Suppose that $R$ is an integral domain.

Define $c_S(\Fdot):=\sum_{k\in\Z} (-1)^{k}\rank(m^k_S(\Fdot))$, and define the {\bf characteristic cycle of $\Fdot$ (in $T^*\U$)} to be the analytic cycle
$$\cc(\Fdot)=\sum_{S\in\strat}c_S(\Fdot)\left[\overline{T^*_S\U}\right].
$$
We write $c_\0(\Fdot)$ in place of $c_{\{\0\}}(\Fdot)$, and let $c_\0(\Fdot)=0$ if $\{\0\}\not\in\strat$.

\medskip

The underlying set $\left|\cc(\Fdot)\right|=\bigcup_{c_S(\Fdot)\neq 0}\overline{T^*_S\U}$ is the {\bf characteristic variety of $\Fdot$ (in $T^*\U$)}.

\end{defn}

\smallskip

Throughout this paper, whenever we refer to $c_S(\Fdot)$ or $\cc(\Fdot)$, we assume that the base ring is an integral domain, even if we do not explicitly state this.

\smallskip

\begin{rem}  We should remark that there are various conventions for the signs involved in the characteristic cycle. In fact, our definition above uses a different convention than we used in our earlier works. Our definition above is the most desirable considering the graded, enriched characteristic cycle that we will define below. In hopes of avoiding confusion with our earlier work, we have also changed our notation for the characteristic cycle.

Note that, using the above convention, the characteristic cycle is not changed by extending $\Fdot$ by zero to all of $\U$.
\end{rem}

\medskip

We give some basic, easy properties of the characteristic cycle concern how they work with shifting, constant sheaves, distinguished triangles, and the Verdier dual $\vdual\Fdot$. The proofs are all trivial, and we leave them to the reader.

\vbox{\begin{prop}\label{prop:basicccprops}

\ 

\begin{enumerate}
\item $\cc(\Fdot[j])=(-1)^j\cc(\Fdot)$.

\smallskip

\item If $X$ is a pure-dimensional (e.g., connected) complex manifold, then 
$$\cc(\mathbf R^\bullet_X) = (-1)^{\dim X}[T^*_X\U],$$
 i.e., $\cc(\mathbf R^\bullet_X[\dim X]) =[T^*_X\U]$.

\smallskip

\item If $\Adot\rightarrow\Bdot\rightarrow\Cdot\arrow{[1]}\Adot$ is a distinguished triangle in $D^b_c(X)$, then $\cc(\Bdot)=\cc(\Adot)+\cc(\Cdot)$.

\smallskip

\item $\cc(\Fdot) = \cc(\vdual\Fdot)$.
\end{enumerate}
\end{prop}}

For calculating the characteristic cycle of the constant sheaf, the following is very useful:

\begin{cor} Suppose that $Y$ and $Z$ are closed analytic subsets of $X$ such that $X=Y\cup Z$. Then, 
$$\cc(\mathbf R^\bullet_X)=\cc(\mathbf R^\bullet_Y)+\cc(\mathbf R^\bullet_Z)-\cc(\mathbf R^\bullet_{Y\cap Z}).
$$
\end{cor}
\begin{proof} Let $j:Y\hookrightarrow X$, $k:Z\hookrightarrow X$, and $l:Y\cap Z\hookrightarrow X$  denote the respective inclusions. Then, there is a canonical distinguished triangle
$$
\mathbf R^\bullet_X\rightarrow j_*j^*\mathbf R^\bullet_X\oplus k_*k^*\mathbf R^\bullet_X\rightarrow l_*l^*\mathbf R^\bullet_X\arrow{[1]}\mathbf R^\bullet_X.
$$
As the pull-back of the constant sheaf is the constant sheaf, and as the characteristic cycle is unaffected by extensions by zero, the desired conclusion follows immediately from Item 3 of \propref{prop:basicccprops}.
\end{proof}

\medskip

We also have the following result.

\begin{prop} Suppose that $\Adot$ and $\Bdot$ are bounded, constructible complexes of sheaves on the $d$-dimensional analytic space $X$. Suppose that $\strat$ is a stratification with respect to which both $\Adot$ and $\Bdot$ are constructible (which always exists). 

For $0\leq k\leq d$, let $\cc_{\geq k}(\Adot)$ and $\cc_{\geq k}(\Bdot)$ denote the portions of the characteristic cycles which lie over closures of strata of dimension greater than or equal to $k$.

Then, $\cc_{\geq k}(\Adot)=\cc_{\geq k}(\Bdot)$  if and only if, for all $S\in \strat$ such that $\dim S\geq k$, for all $\mathbf p\in S$, there is an equality of Euler characteristics of the stalk cohomology $\chi(\Adot)_{\mathbf p}=\chi(\Bdot)_{\mathbf p}$.

In particular, $\cc(\Adot)=\cc(\Bdot)$  if and only if, for all $\mathbf p\in X$, $\chi(\Adot)_{\mathbf p}=\chi(\Bdot)_{\mathbf p}$.
\end{prop}
\begin{proof} The proof is by downward induction on $k$. Certainly the result is trivial for $k=d$. Now suppose that $k_0\geq 0$ and that the statement is true for all $k$ such that $k_0+1\leq k\leq d$; we wish to show that the statement is true for $k=k_0$.

Let $S_0\in\strat$ be a stratum of dimension $k_0$, and let $\mathbf p_0\in S_0$. For each stratum $S$ of dimension greater than or equal to $k_0+1$, let $\p_S$ denote a point of $S$. If we let be $\Adot$ or $\Bdot$, then
$$c_{S_0}(\Fdot):=\chi(\N_{S_0}, \cL_{S_0}; \Fdot[-k_0])= \chi(\N_{S_0}; \Fdot[-k_0])-\chi( \cL_{S_0}; \Fdot[-k_0])=$$
$$
(-1)^{k_0}\Big\{\chi(\Fdot)_{\p_0} -\sum_{S, \dim S\geq k_0+1}\chi\big(\cL_{S_0}\cap S\big)\cdot\chi(\Fdot)_{\p_S} \Big\}.
$$
Note that our inductive hypothesis implies that the summation on the right above is the same whether $\Fdot$ equals $\Adot$ or $\Bdot$.

Therefore, $c_{S_0}(\Adot)= c_{S_0}(\Bdot)$ if and only if $\chi(\Adot)_{\p_0}=\chi(\Bdot)_{\p_0}$, and we are finished.
\end{proof}

\smallskip

\begin{cor}\label{cor:ccvan} Suppose that $\cc(\Adot)=\cc(\Bdot)$ and that we have a complex analytic $f:X\rightarrow\C$. Then, for all, $\mathbf p\in X$,
$$
\chi(\psi_{f-f(\mathbf(p)}\Adot)_{\mathbf p}=\chi(\psi_{f-f(\mathbf(p)}\Bdot)_{\mathbf p}\hskip 0.2in\textnormal{and}\hskip 0.2in \chi(\phi_{f-f(\mathbf(p)}\Adot)_{\mathbf p}=\chi(\phi_{f-f(\mathbf(p)}\Bdot)_{\mathbf p}.
$$
\end{cor}
\begin{proof} For convenience, we shall assume that $f(\p)=0$. Let $F_{f, \p}$ denote the Milnor fiber of $f$ at $\p$. Once again, choose a Whitney stratification $\strat$ with respect to which both $\Adot$ and $\Bdot$ are constructible and, for each $S\in\strat$, select a $\p_S\in S$.

Then,
$$
\chi(\psi_f\Adot)_{\mathbf p}=\chi(F_{f, \p}; \Adot)=\sum_{S\in\strat}\chi(F_{f,\p}\cap S)\cdot\chi(\Adot)_{\p_S}.
$$
By the proposition, this also equals $\chi(\psi_f\Bdot)_{\mathbf p}$.

The result about the vanishing cycles follows immediately since
$$
 \chi(\phi_f\Adot)_{\mathbf p}= \chi(\psi_f\Adot)_{\mathbf p}-\chi(\Adot)_\p = \chi(\psi_f\Bdot)_{\mathbf p}-\chi(\Bdot)_\p= \chi(\phi_f\Bdot)_{\mathbf p}.
$$
\end{proof}

\vskip 0.3in

The characteristic cycle uses only the Euler characteristic information from the Morse data to strata. While this makes many calculations far easier, it disposes of a great deal of cohomological data. Hence, we define a formal graded  ``cycle'' with module coefficients (actually, isomorphism classes of modules); we shall discuss such ``enriched'' cycles more generally in \secref{sec:enrcycles}. 

\begin{defn}\label{def:gecc} The {\bf graded, enriched characteristic cycle of $\Fdot$ in the cotangent bundle $T^*\U$} is defined in degree to be
$k$ to be 
$$
\gecc^k(\Fdot) :=\sum_{S\in\strat} m_S^k(\Fdot)\big[\,\overline{T^*_{{}_{S}}\U}\,\big] =  \sum_{S\in\strat} H^{k-d_S}(\mathbb N_S,\mathbb L_S; \Fdot)\big[\,\overline{T^*_{{}_{S}}\U}\,\big].
$$
The underlying set $\left|\gecc^k(\Fdot)\right|$ is $\bigcup_{m^k_S(\Fdot)\neq 0}\overline{T^*_S\U}$. 

\medskip

The total underlying set $\left|\gecc^\bullet(\Fdot)\right|:=\bigcup_{m^*_S(\Fdot)\neq 0}\overline{T^*_S\U}$ is the {\bf microsupport of $\Fdot$ (in $T^*\U$)}, and is denoted by $\ms(\Fdot)$ \textnormal{(see \cite{kashsch})}.
\end{defn}

 \medskip

We have the following basic properties.

\begin{prop}\label{prop:basicgecc}
\ 

\begin{enumerate}
\item $\gecc^k(\Fdot[j])=\gecc^{k+j}(\Fdot)$.

\smallskip

\item $\supp\Fdot=\pi(\ms(\Fdot))$.

\smallskip

\item If $X$ is a pure-dimensional (e.g., connected) complex manifold, then 
$$\gecc^k(\mathbf R^\bullet_X[\dim X]) =
\begin{cases}
0, \textnormal{ if } k\neq 0;\\
 R[T^*_X\U], \textnormal{ if } k=0.
\end{cases}
$$

\smallskip

\item More generally, $\Fdot$ is a perverse sheaf if and only if $\gecc^\bullet(\Fdot)$ is concentrated in degree $0$, i.e., $\gecc^k(\Fdot) =0$ if $k\neq 0$. In particular, the characteristic cycle of a perverse sheaf has only non-negative coefficients.

\smallskip

\item Let ${}^{\mu}\hskip -.02in H^k$ denote the degree $k$ (middle perversity) perverse cohomology (see\cite{kashsch}, section 10.3). Then, ${}^{\mu}\hskip -.02in H^k(\Fdot)$ is a perverse sheaf, $m_S^k(\Fdot)\cong m_S^0({}^{\mu}\hskip -.02in H^k(\Fdot))$, and 
$$
\gecc^k(\Fdot)=\gecc^0\big({}^{\mu}\hskip -.02in H^k(\Fdot)\big).
$$

\smallskip

\item Suppose that $\Adot\rightarrow\Bdot\rightarrow\Cdot\arrow{[1]}\Adot$ is a distinguished triangle in $D^b_c(X)$. Then, for all $k$, $|\gecc^k(\Bdot)|\subseteq |\gecc^k(\Adot)|\cup|\gecc^k(\Cdot)|$ and, consequently, the microsupport of each complex is contained in the union of the microsupports of the other two.

\smallskip

\item If $R$ is Dedekind domain, then, for all $k$ and for all $S\in\strat$, 
$$m^{-k}_S(\vdual\Fdot)\cong \textnormal{Hom}(m^k_S(\Fdot), R)\oplus \textnormal{Ext}(m^{k+1}_S(\Fdot), R).$$
In particular, if $R$ is a field, then $\gecc^{-k}(\vdual\Fdot)=\gecc^k(\Fdot)$.

\smallskip

\item Suppose that $R$ is a principal ideal domain. Let $X$ and $Y$ be analytic spaces, let $\pi_1:X\times Y\rightarrow X$ and $\pi_2:X\rightarrow Y$ denote the projections. Let $\strat$ and $\strat'$ be Whitney stratifications of $X$ and $Y$, respectively. Let $\Adot$ and $\Bdot$ be bounded,  complexes of sheaves on $X$ and $Y$, respectively, which are constructible with respect to $\strat$ and $\strat'$, respectively. Let $\Adot\lboxtimes\Bdot:=\piten$.

Then, $\Adot\lboxtimes\Bdot$ is constructible with respect to the product stratification $\{S\times S'\ | \ S\in \strat, \ S'\in\strat'\}$ and, for all $S\in\strat$ and $S'\in \strat'$,

\medskip

$m_{S\times S'}^k\big(\Adot\lboxtimes\Bdot\big) =\hfill$

\smallskip

$\hfill\displaystyle\bigoplus_{i+j=k}m^i_S(\Adot)\otimes m^j_{S'}(\Bdot)  \ \oplus \ \bigoplus_{i+j=k+1}\operatorname{Tor}\big(m^i_S(\Adot), m^j_{S'}(\Bdot)\big)$.

\medskip

Consequently,
$$
c_{S\times S'}\big(\Adot\lboxtimes\Bdot\big)= c_S(\Adot)\cdot c_{S'}(\Bdot).
$$

\end{enumerate}
\end{prop}

\begin{proof} Items 1, 3, and 6 are trivial. Item 2 is the last equality of Proposition 2.5 in \cite{singenrich}. Item 8 follows immediately from formula 5.6 of \cite{schurbook}.

\medskip 

To see Item 4, note that the graded, enriched characteristic cycle of a complex being concentrated in degree zero is equivalent to the complex being {\bf pure with shift $0$} (see Definition 7.5.4 of \cite{kashsch}). This is equivalent to the complex being perverse (\cite{kashsch}, 9.5.2).

\medskip

\noindent Item 5:

We will use that $\phi_f[-1]$ naturally commutes with ${}^{\mu}\hskip -.02in H^k$ (\cite{kashsch}, Corollary 10.3.13) and that, if $\Adot$ has $\mathbf p$ as an isolated point in its support, then so does the perverse sheaf ${}^{\mu}\hskip -.02in H^k(\Adot)$; in this case, the stalk cohomology of ${}^{\mu}\hskip -.02in H^k(\Adot)$ at $\mathbf p$ is concentrated in degree $0$ and 
$$
H^0\left({}^{\mu}\hskip -.02in H^k(\Adot)\right)_{\mathbf p}\cong H^k(\Adot)_{\mathbf p}.
$$ 

Now, let $S\in\strat$. Let $\p\in S$. Let $g$ be as in \remref{rem:morsemod}, so that $\mathbf p$ is an isolated point in the support of $\phi_g[-1]\Fdot$ and, hence,
$$m_S^k(\Fdot)\cong H^k(\phi_g[-1]\Fdot)_{\p}\cong H^0\left({}^{\mu}\hskip -.02in H^k(\phi_g[-1]\Fdot)\right)_{\mathbf p}\cong $$
$$H^0\left(\phi_g[-1]{}^{\mu}\hskip -.02in H^k(\Fdot)\right)_{\mathbf p}\cong m_S^0({}^{\mu}\hskip -.02in H^k(\Fdot)).$$
\medskip

\noindent Item 7:

We begin as in the proof of Item 5. Let $S\in\strat$. Let $\p\in S$. Let $g$ be as in \remref{rem:morsemod}, so that 
$$m_S^k(\Fdot)\cong H^k(\phi_g[-1]\Fdot)_{\p},  \ \ m_S^{k+1}(\Fdot)\cong H^{k+1}(\phi_g[-1]\Fdot)_{\p},$$ and 
$$m_S^{-k}(\vdual\Fdot)\cong H^{-k}(\phi_g[-1]\vdual\Fdot)_{\p}\cong H^{-k}(\vdual\phi_g[-1]\Fdot)_{\p}.$$
As the support of $\phi_g[-1]\Fdot$ is contained in $\{\p\}$ and as $R$ is a Dedekind domain, there is a natural split exact sequence

\smallskip

$0 \rightarrow \textnormal{Ext}(H^{k+1}(\phi_g[-1]\Fdot)_\p , R) \rightarrow 
H^{-k}(\vdual\phi_g[-1]\Fdot)_\p \rightarrow\hfill$

\smallskip

$\hfill \textnormal{Hom}(H^k(\phi_g[-1]\Fdot)_\p, R) \rightarrow 0.$

Item 7 follows.

\end{proof}

\medskip

Note that, for a perverse sheaf $\Pdot$ with free Morse modules, $\gecc^\bullet(\Pdot)$ is completely determined by $\cc(\Pdot)$. This motivates us to define:

\begin{defn}\label{def:numerical} A complex of sheaves $\Pdot$ is {\bf numerical} if and only if $\Pdot$ is perverse with free Morse modules.
\end{defn}

\smallskip

\noindent\rule{1in}{1pt}

\bigskip

While most of our examples will have to wait until we have developed more machinery, we can calculate ``bare-handedly'' what happens for curves and some basic complexes of sheaves.

\begin{exm}\label{exm:curvegecc} Suppose that $X$ is a curve. The calculations of $\cc(\Fdot)$ and $\gecc^\bullet(\Fdot)$ reduce to calculating what happens at the discrete set of points where $X$ is singular or where $\Fdot$ is not locally constant. Thus, it suffices to analyze the situation where there is a single zero-dimensional stratum.

Hence, we shall assume that $\0\in X$, and that $X-\{\0\}$ is a smooth curve.  In a small enough open ball, each irreducible component of $X$ is homeomorphic to an open disk and, hence, corresponds to an irreducible component of the germ of $X$ at $\0$. Let $\{X_\mfl\}_{\mfl\in \Lambda}$ denote the collection of irreducible components of $X$ and, for each $X_\mfl$, let $\mfm_\mfl:={\operatorname{mult}}_\0 X_\mfl$. Let $\mfm:={\operatorname{mult}}_\0 X=\sum_{\mfl}\mfm_\mfl$. Let $e:=|\Lambda|$, i.e., let $e$ be the number of irreducible components of $X$.

Stratify $X$ by using $S_0:=\{\0\}$ and $S_\mfl:=X_\mfl-\{\0\}$ as strata (we are assuming that $0$ is not in the indexing set $\Lambda$). Let $j:\{\0\}\hookrightarrow X$ and $i:X-\{\0\}\hookrightarrow X$ denote the inclusions. Let $\Adot:=\Z^\bullet_X[1]$, $\Bdot:=i_!i^!\Adot$, $\Cdot:=i_*i^*\Adot$, and let $\Idot$ be the perverse sheaf given by intersection cohomology with constant $\Z$-coefficients (here, we use the shifts that put all of the possibly non-zero cohomology in non-positive degrees). These are all complexes of sheaves on $X$, which are the constant sheaf, shifted by $1$, on $X-\{\0\}$.

Let $\tilde f:(\U, \0)\rightarrow(\C, 0)$ be a complex analytic function, where $\tilde f$ may vanish identically on some irreducible components of $X$. Let $f:={\tilde f}_{|_{X}}$. Consider $\Pdot:=\psi_f[-1]\Adot$ and $\Qdot:=\phi_f[-1]\Adot$. These are complexes of sheaves on $V(f):=f^{-1}(0)$.

\medskip

We wish to calculate the graded, enriched characteristic cycle, and the ordinary characteristic cycle, for each of the six complexes given above.

\medskip

As $\Adot$, $\Bdot$, $\Cdot$, and $\Idot$ are the $1$-shifted constant sheaf on $X-\{\0\}$, it follows that, if $\Fdot$ is any of these four complexes, then, for all $\mfl\in \Lambda$, $m_{S_\mfl}^0(\Fdot) \cong \Z$, and $m_{S_\mfl}^k(\Fdot) =0$ for $k\neq 0$. The question is: what is $m_{S_0}^k(\Fdot)$?

A normal slice to $S_0$ is simply $\overset{\circ}{B_\epsilon}\cap X$, where $\overset{\circ}{B_\epsilon}$ is a small open ball around the origin. The complex link to $S_0$ is $\overset{\circ}{B_\epsilon}\cap X\cap L^{-1}(a)$, where $L$ is a generic linear form and $0<|a|\ll\epsilon$. 

\vskip .2in

\noindent $\Adot$:

We have
$$
m_{S_0}^k(\Adot)= \hyp^{k}(\overset{\circ}{B_\epsilon}\cap X,\overset{\circ}{B_\epsilon}\cap X\cap L^{-1}(a); \Z^\bullet_X[1]),
$$
which is the ordinary degree $k+1$ integral cohomology of the pair consisting of a contractible space modulo $\mfm$ points in the space. Hence, $m_{S_0}^0(\Adot)\cong \Z^{(\mfm-1)}$, and $m_{S_0}^k(\Adot)=0$ if $k\neq 0$. 

Thus, we find that, if $k\neq 0$, then $\gecc^k(\Adot) = 0$, and
$$
\gecc^0(\Adot) = \Z^{(\mfm-1)}\left[T_\0^*\U\right]+\sum_\mfl \Z\left[\overline{T^*_{X_\mfl-\{\0\}}\U}\right].
$$
It follows that
$$
\cc(\Adot) = (\mfm-1)\left[T_\0^*\U\right]+\sum_\mfl \left[\overline{T^*_{X_\mfl-\{\0\}}\U}\right].
$$

\vskip .3in

\noindent $\Bdot$:

We have
$$
m_{S_0}^k(\Bdot)= \hyp^{k+1}(\overset{\circ}{B_\epsilon}\cap X,\overset{\circ}{B_\epsilon}\cap X\cap L^{-1}(a); i_!\Z^\bullet_{X-\{\0\}}).
$$
Using the long exact sequence for the hypercohomology of a pair, and that 
$$\hyp^*(\overset{\circ}{B_\epsilon}\cap X; i_!\Z^\bullet_{X-\{\0\}})=0,$$ we find that
$$
m_{S_0}^k(\Bdot)\cong H^{k}(\overset{\circ}{B_\epsilon}\cap X\cap L^{-1}(a); \Z).
$$
Therefore, $m_{S_0}^k(\Bdot)=0$ if $k\neq 0$, and $m_{S_0}^0(\Bdot)\cong\Z^\mfm$.

\smallskip

Thus, we find that, if $k\neq 0$, then $\gecc^k(\Bdot) = 0$, and
$$
\gecc^0(\Bdot) = \Z^{\mfm}\left[T_\0^*\U\right]+\sum_\mfl \Z\left[\overline{T^*_{X_\mfl-\{\0\}}\U}\right].
$$
It follows that
$$
\cc(\Bdot) = \mfm\left[T_\0^*\U\right]+\sum_\mfl \left[\overline{T^*_{X_\mfl-\{\0\}}\U}\right].
$$

\vskip .3in

\noindent $\Cdot$:

We have
$$
m_{S_0}^k(\Cdot)= \hyp^{k+1}(\overset{\circ}{B_\epsilon}\cap X,\overset{\circ}{B_\epsilon}\cap X\cap L^{-1}(a); i_*\Z^\bullet_{X-\{\0\}})\cong$$
$$H^{k+1}(\overset{\circ}{B_\epsilon}\cap X-\{\0\},\overset{\circ}{B_\epsilon}\cap X\cap L^{-1}(a); \Z).
$$
This splits as a direct sum of the degree $k+1$ integral cohomology of pairs consisting of spaces $\overset{\circ}{B_\epsilon}\cap X_\mfl-\{\0\}$, which are homotopy-equivalent to circles, modulo $\mfm_\mfl$ points. As in the $\Bdot$ case, one easily calculates that that $m_{S_0}^k(\Cdot)=0$ if $k\neq 0$, and $m_{S_0}^0(\Cdot)\cong\Z^\mfm$.

Thus, we find that $\gecc^\bullet(\Bdot)=\gecc^\bullet(\Cdot)$ and, of course, that $\cc(\Bdot)=\cc(\Cdot)$.

\vskip .3in

\noindent $\Idot$:

The axioms of intersection cohomology imply that $\Idot$ is isomorphic to the direct sum of the extensions by zero of the intersection cohomology on each of the $X_\mfl$. As each $X_\mfl$ is homeomorphic to an open disk, the intersection cohomology complex on $X_\mfl$ is isomorphic to $\Z_{X_\mfl}[1]$.

Thus, we have
$$
m_{S_0}^k(\Idot)= \bigoplus_{\mfl\in\Lambda}H^{k+1}(\overset{\circ}{B_\epsilon}\cap X_\mfl,\overset{\circ}{B_\epsilon}\cap X_\mfl\cap L^{-1}(a); \Z).$$
It follows that $m_{S_0}^k(\Idot)=0$ if $k\neq 0$, and $m_{S_0}^0(\Idot)\cong\bigoplus_\mfl \Z^{(\mfm_\mfl-1)}\cong\Z^{(\mfm-e)}$.

Thus, we find that, if $k\neq 0$, then $\gecc^k(\Idot) = 0$, and
$$
\gecc^0(\Idot) = \Z^{(\mfm-e)}\left[T_\0^*\U\right]+\sum_\mfl \Z\left[\overline{T^*_{X_\mfl-\{\0\}}\U}\right].
$$
It follows that
$$
\cc(\Idot) = (\mfm-e)\left[T_\0^*\U\right]+\sum_\mfl \left[\overline{T^*_{X_\mfl-\{\0\}}\U}\right].
$$

\medskip

We still wish to look at the complexes $\Pdot$ and $\Qdot$. Let $\Lambda_{\subseteq}:=\{\mfl\in\Lambda\ |\ f(X_\mfl)\equiv 0\}$. Hence, $V(f)=\bigcup_{\mfl\in \Lambda_{\subseteq}}X_\mfl$, with the convention that, if $\Lambda_{\subseteq}$ is empty, then this union is taken as meaning the point-set $\{\0\}$. Let $\Lambda_{\not\subseteq}:=\Lambda-\Lambda_{\subseteq}$. For each $\mfl\in\Lambda_{\not\subseteq}$, let $\eta_\mfl$ equal the intersection multiplicity $(X_\mfl\cdot V(\tilde f))_\0$, and let $\eta:=\sum_{\mfl\in\Lambda_{\not\subseteq}}\eta_\mfl$. Let $\mfm_\subseteq:=\sum_{\mfl\in\Lambda_\subseteq}\mfm_\mfl$.

\vskip .3in

\noindent $\Pdot$:

By definition, $\Pdot=\psi_f[-1]\Adot$ is a complex of sheaves on $V(f)$, but -- in our current setting -- the support of $\psi_f[-1]\Adot$ will be contained in $\{\0\}$, and the stalk cohomology $H^*(\psi_f[-1]\Adot)_\0$ is isomorphic to $\Z^\eta$ in degree $0$ and is zero in other degrees.

Thus, we find that, if $k\neq 0$, then $\gecc^k(\Pdot) = 0$, and
$$
\gecc^0(\Pdot) = \Z^{\eta}\left[T_\0^*\U\right].
$$
It follows that
$$
\cc(\Pdot) = \eta\left[T_\0^*\U\right].
$$

\vskip .3in

\noindent $\Qdot$:

By definition, $\Qdot=\phi_f[-1]\Adot$ is a complex of sheaves on $V(f)$, and the restriction of $\Qdot$ to $V(f)-\{\0\}$ is isomorphic to the $1$-shifted constant sheaf. In addition, the stalk cohomology $H^*(\phi_f[-1]\Adot)_\0$ is isomorphic to $\Z^{(\eta-1)}$ in degree $0$ and is zero in other degrees. We have
$$
m_{S_0}^k(\Qdot)= \hyp^{k}(\overset{\circ}{B_\epsilon}\cap X\cap V(f),\overset{\circ}{B_\epsilon}\cap X\cap V(f)\cap L^{-1}(a);\phi_f[-1]\Adot),
$$
for $0<|a|\ll 1$. This module fits into the hypercohomology long exact sequence of the pair, in which one has the map induced by inclusion
$$
\hyp^{k}(\overset{\circ}{B_\epsilon}\cap X\cap V(f);\phi_f[-1]\Adot)\rightarrow \hyp^{k}(\overset{\circ}{B_\epsilon}\cap X\cap V(f)\cap L^{-1}(a);\phi_f[-1]\Adot).\eqno{(\dagger)}
$$
The right-hand term above is clearly isomorphic to $$\bigoplus_{\mfl\in\Lambda_\subseteq}\hyp^k(\overset{\circ}{B_\epsilon}\cap X_\mfl\cap L^{-1}(a);\phi_f[-1]\Adot),$$
and $(\dagger)$ can be rewritten as
$$
H^{k+1}(\overset{\circ}{B_\epsilon}\cap X, \overset{\circ}{B_\epsilon}\cap X\cap f^{-1}(b);\Z)\rightarrow \bigoplus_{\mfl\in\Lambda_\subseteq}H^{k+1}(\overset{\circ}{B_\epsilon}\cap X_\mfl\cap L^{-1}(a);\Z),
$$
where $0<|b|\ll |a|\ll 1$. Now, one easily finds from the long exact sequence that
$m_{S_0}^k(\Qdot)=0$ for $k\neq 0$, and $m_{S_0}^0(\Qdot)\cong\Z^{(\mfm_\subseteq+\eta-1)}$.

Thus, we find that, if $k\neq 0$, then $\gecc^k(\Qdot) = 0$, and
$$
\gecc^0(\Qdot) = \Z^{(\mfm_\subseteq+\eta-1)}\left[T_\0^*\U\right]+\sum_{\mfl\in\Lambda_{\subseteq}} \Z\left[\overline{T^*_{X_\mfl-\{\0\}}\U}\right].
$$
It follows that
$$
\cc(\Qdot) = (\mfm_\subseteq+\eta-1)\left[T_\0^*\U\right]+\sum_{\mfl\in\Lambda_{\subseteq}} \left[\overline{T^*_{X_\mfl-\{\0\}}\U}\right].
$$

\vskip .4in

The characteristic cycle calculations for $\Pdot$ and $\Qdot$ can be ``checked''. There is the canonical distinguished triangle
$$
(\Z^\bullet_X[1])_{V(f)}[-1]\rightarrow \psi_f[-1]\Z^\bullet_X[1]\rightarrow \phi_f[-1]\Z^\bullet_X[1]\arrow{[1]}(\Z^\bullet_X[1])_{V(f)}[-1],
$$
and so we should find that $\cc(\Pdot)=\cc(\Qdot)+\cc(\Z^\bullet_{V(f)})$. 

This is easily checked, for $\cc(\Z^\bullet_{V(f)})= -\cc(\Z^\bullet_{V(f)}[1])$ and, applying our calculation of $\cc(\Adot)$, we find that 
$$\cc(\Z^\bullet_{V(f)}[1]) = (\mfm_\subseteq-1)\left[T_\0^*\U\right] +\sum_{\mfl\in\Lambda_{\subseteq}} \left[\overline{T^*_{X_\mfl-\{\0\}}\U}\right].
$$

\bigskip

Note that all of the graded, enriched characteristic cycles in this example are concentrated in degree $0$. As we stated in Item 4 of \propref{prop:basicgecc}, this is a reflection of the fact that each of the complexes that we considered above are perverse sheaves. 
\end{exm}

\begin{exm} We wish a give an easy example/problem, where the sheaves under consideration are not perverse.

\smallskip

Let $\U:=\C^3$, and use $x$, $y$, and $z$ as coordinates. Let $X:=V(z)\cup V(x, y)$. There are three obvious strata: $S_0:=\{\0\}$, $S_1:=V(x, y)-\{\0\}$, and $S_2:= V(z)-\{\0\}$. Let $j:\{\0\}\hookrightarrow X$ and $i:X-\{\0\}\hookrightarrow X$ denote the inclusions, and consider the complexes of sheaves $\Adot:=\Z^\bullet_X[2]$, $\Bdot:=i_!i^!\Adot$, and $\Cdot:=i_*i^*\Adot$.

\bigskip

We leave it to the reader to verify that:

\bigskip

\noindent $\gecc^k(\Adot)=0$ if $k\neq -1, 0$, 

\smallskip

\noindent $\gecc^{-1}(\Adot)=\Z[T^*_{\0}\U]+\Z[T^*_{V(x, y)}\U]$, and $\gecc^0(\Adot)=\Z[T^*_{V(z)}\U]$.

\bigskip

\noindent  $\gecc^k(\Bdot)=0$ if $k\neq -1, 0$, 

\smallskip

\noindent $\gecc^{-1}(\Bdot)=\Z^2[T^*_{\0}\U]+\Z[T^*_{V(x, y)}\U]$, and $\gecc^0(\Bdot)=\Z[T^*_{V(z)}\U]$.

\bigskip

\noindent $\gecc^k(\Cdot)=0$ if $k\neq -1, 0, 1$, 

\smallskip

\noindent $\gecc^{-1}(\Cdot)=\Z[T^*_{\0}\U]+\Z[T^*_{V(x, y)}\U]$, $\gecc^0(\Cdot)=\Z[T^*_{V(z)}\U]$, and $\gecc^{1}(\Cdot)=\Z[T^*_{\0}\U]$.
\end{exm}

\medskip

\begin{rem} Before we leave this section, we need to make an important point.

In this paper, we shall show that many results about characteristic cycles are true for graded, enriched characteristic cycles -- with essentially the same proofs; however, because the graded, enriched characteristic cycles contain far more data, it should not be surprising that, for some results, the ordinary characteristic cycle is easier to calculate with. In particular, the additivity of CC for distinguished triangles, as given in Item 3 of \propref{prop:basicccprops}, is extremely useful.
\end{rem}

\section{Characteristic Complexes and the Local Euler Obstruction}\label{sec:localeuler}

Throughout this paper, we emphasize that many classical problems on the local topology of hypersurfaces, inside a possibly singular space $X$, can/should be approached by taking various complexes of sheaves on $X$ and looking at their characteristic cycles or graded, enriched characteristic cycles. 

However, some classical constructions, such as calculating the polar varieties and polar multiplicities of L\^e and Teissier, deal with contributions from only the smooth strata of $X$. From our point of view, these are results where the underlying complex of sheaves is a {\it characteristic complex}.

\begin{defn} Let $X=\bigcup_i X_i$ be the decomposition of $X$ into its irreducible components. 

We say that a complex of sheaves $\Kdot$ on $X$ is a {\bf characteristic complex for $X$} provided that
$$
\cc(\Kdot) = \left[\overline{T^*_{X_{\operatorname{reg}}}\U}\right]= \sum_i \left[\overline{T^*_{(X_i)_{\operatorname{reg}}}\U}\right].
$$
\end{defn}

\medskip

\begin{prop}\label{prop:ccc} Let $X=\bigcup_i X_i$ be the decomposition of $X$ into its irreducible components.  Suppose that, for each $i$, $\Kdot_i$ is a characteristic complex for $X_i$ and let ${\widehat{\mathbf K}}_i^\bullet$ denote the extension by zero of $\Kdot_i$ to all of $X$. Then, $\bigoplus_i {\widehat{\mathbf K}}_i^\bullet$ is a characteristic complex for $X$.
\end{prop}
\begin{proof} This is immediate from Item 3 of \propref{prop:basicccprops}.
\end{proof}

\medskip

\begin{prop} Characteristic complexes exist for all $X$.
\end{prop}
\begin{proof} This proof is contained in Lemma 3.1 of \cite{hypercohom}. However, we wish to sketch it here.

Note that \propref{prop:ccc} implies that we need deal only with the case where $X$ is irreducible. Hence, we assume that $X$ is irreducible of dimension $d$.

\smallskip

Let $\strat$ be a Whitney stratification of $X$, with connected strata. Recall that our base ring is $R$, which we are assuming is an integral domain.

\smallskip

For every stratum $S\in \strat$ and every non-negative integer $v$, let $\Udot_{S, v}$ denote the extension by zero to all of $X$ of $\big(\mathbf R^\bullet_S\big)^v[\dim S]$ so that $c_S(\Udot_{S,v})=v$ (where $c_S$ is the coefficient of $\big[\overline{T_S^*\U}\big]$ in the characteristic cycle). If $v$ is a negative integer, define $\Udot_{S,v}:=\Udot_{S,-v}[1]$ so that, again,  $c_S(\Udot_{S,v})=v$.

Now we construct a characteristic complex as a direct sum, canceling out conormal cycles over lower-dimensional strata.  Let 
$$\Kdot_d=\Kdot_{\geq d} :=\Udot_{X_{\operatorname{reg}}, 1}.
$$
Note that $\Kdot_d$ is also constructible with respect to $\strat$ and, if $S\in \strat$ has dimension $d$, then $c_S(\Kdot_d)=1$.

Now we need cancel out the contributions to the characteristic cycle from lower-dimensional strata. 

Let 
$$\Kdot_{d-1}:=\bigoplus_{S\in \strat, \dim S=d-1}\Udot_{S,-c_S(\Cdot_d)},$$
so that $\Kdot_{\geq d-1}:=\Kdot_d\oplus\Kdot_{d-1}$ has the property that, for $S\in \strat$ of dimension at least $d-1$,
$$
c_S\left(\Kdot_{\geq d-1}\right)=
\begin{cases}
1, \textnormal{ if } \dim S=d;\\
0, \textnormal{ if }\dim S=d-1.
\end{cases}
$$

Continuing in this manner, we produce $\Kdot:=\Kdot_{\geq 0}$ which is a characteristic complex for $X$.
\end{proof}

\medskip

\begin{prop}\label{prop:kextprod} Suppose that $R$ is a principal ideal domain. Let $X$ and $Y$ be analytic spaces, let $\pi_1:X\times Y\rightarrow X$ and $\pi_2:X\rightarrow Y$ denote the projections.  Let $\Kdot_X$ and $\Kdot_Y$ be characteristic complexes for $X$ and $Y$, respectively. Let $\Kdot_X\lboxtimes\Kdot_Y:=\pi_1^*\Kdot_X\lotimes\pi_2^*\Kdot_Y$.

Then, $\Kdot_X\lboxtimes\Kdot_Y$ is a characteristic complex for $X\times Y$.

\end{prop}
\begin{proof} This is immediate from Item 8 of \propref{prop:basicgecc}.
\end{proof}

\smallskip

\noindent\rule{1in}{1pt}

\bigskip

Aside from giving us a way, in terms of constructible complexes, to isolate contributions from smooth strata, our primary interest in characteristic complexes lies in their relationship with the famous {\it local Euler obstruction} of MacPherson \cite{maceuler}. We let $\operatorname{Eu}_{\mathbf p}X$ denote the local Euler obstruction of $X$ at $\mathbf p$.

Basic properties of the local Euler obstruction are:

\begin{enumerate}
\item The local Euler obstruction is, in fact, local, i.e., if $\W$ is an open neighborhood of $\mathbf p$ in $X$, then $\operatorname{Eu}_{\mathbf p}X=\operatorname{Eu}_{\mathbf p}\W.$
\smallskip
\item If $\mathbf p$ is a smooth point of $X$, then $\operatorname{Eu}_{\mathbf p}X=1$.
\smallskip
\item If $(\mathbf x, \mathbf y)\in X\times Y$, then $\operatorname{Eu}_{(\mathbf x, \mathbf y)}(X\times Y)=\left(\operatorname{Eu}_{\mathbf x}X\right)\left(\operatorname{Eu}_{\mathbf y}Y\right)$.
\smallskip
\item If $\mathbf p\in X$ and $X_i$ denotes the local irreducible components of $X$ at $\mathbf p$, then $\operatorname{Eu}_{\mathbf p}X=\sum_i\operatorname{Eu}_{\mathbf p}X_i.$
\smallskip
\item $\operatorname{Eu}_{\mathbf x}X$ is a constant function of $\mathbf x$ along the strata of any Whitney stratification of $X$ (which has connected strata).
\end{enumerate}

\smallskip

There is also the important result:

\begin{thm}\label{thm:bdk} (Brylinski, Dubson, Kashiwara, \cite{bdk}) Suppose that $\Adot$ on $X$ is constructible with respect a Whitney stratification $\strat$, and that
$$
\cc(\Adot)=\sum_{S\in\strat}c_S(\Adot)\left[\overline{T^*_S\U}\right].
$$
Then, for all $\mathbf p\in X$,
$$
\chi(\Adot)_{\mathbf p}=\sum_{S\in\strat} (-1)^{\dim S}c_S(\Adot)\operatorname{Eu}_{\mathbf p}\overline S,
$$
where we set $\operatorname{Eu}_{\mathbf p}(\overline S)=0$ if $\mathbf p\not\in\overline{S}$.
\end{thm}

\medskip

We immediately conclude:

\begin{cor}\label{cor:eulerob} Let $X=\bigcup_i X_i$ be the decomposition of $X$ into its irreducible components, and let $d_i:=\dim X_i$. Let $\Kdot_i$ be a characteristic complex for $X_i$. Let $\mathbf p\in X$. 

Then,
$$
\operatorname{Eu}_{\mathbf p}X= \sum_i (-1)^{\dim X_i}\chi\big(\Kdot_i\big)_{\mathbf p}.
$$

In particular, if $X$ is pure-dimensional and $\Kdot$ is a characteristic complex for $X$, then
$$
\operatorname{Eu}_{\mathbf p}X= (-1)^{\dim X}\chi\big(\Kdot\big)_{\mathbf p}.
$$
\end{cor}

\smallskip

\noindent\rule{1in}{1pt}

\bigskip
 We wish to discuss the {\it relative local Euler obstruction}, as was introduced  in \cite{bmps}.
 
 \smallskip
 
Suppose that we $\p\in X$ and a complex analytic $f:X\rightarrow\C$. We let $\tilde f$ be a local extension of $f$ at $\p$ to an open neighborhood $\W$ of $\p$ in $\U$. We also let $d\tilde f$ denote the section of the cotangent bundle to $\W$ given by $d\tilde f(\mathbf x)=(\mathbf x, d_{\mathbf x}\tilde f)$; we let $\operatorname{im}(d\tilde f)$ denote the image of this section in $T^*\U$.

\medskip

Assuming that $X$ is pure-dimensional,  the relative local Euler obstruction, $\operatorname{Eu}_{\mathbf p}f$, is defined, provided that $\p$  is a stratified isolated critical point of $f$; see \cite{bmps}.

\medskip

In Corollary 5.4 of \cite{bmps}, we show:

\begin{prop}\label{prop:releuler} Suppose that $X$ is pure-dimensional and that $f:X\rightarrow \C$ has a stratified isolated critical point at $\mathbf p$. Let $\Kdot$ be a characteristic complex for $X$. 
Then, $(\mathbf p, d_{\mathbf p}\tilde f)$ is an isolated point in the intersection $\overline{T_{X_{\operatorname{reg}}}^*\U}\cap\operatorname{im}(d\tilde f)$ and

$$
\operatorname{Eu}_{\mathbf p}f =  (-1)^{\dim X}\chi\big(\phi_{f-f(\mathbf p)}[-1]\Kdot\big)_{\mathbf p}=(-1)^{\dim X}\left(\overline{T^*_{X_{\operatorname{reg}}}\U}\cdot \operatorname{im}(d\tilde f)\right)_{(\mathbf p, d_{\mathbf p}\tilde f)},
$$
where this intersection number, in the case where $X$ is affine space, is the Milnor number of $f-f(\mathbf p)$ at $\mathbf p$.
\end{prop}
Note that the above corollary looks slightly different from what appears in Corollary 5.4 of \cite{bmps}; this is because our definition of the characteristic cycle has changed by a shift.

\medskip

We can use \propref{prop:releuler} as the basis for generalizing the definition of the relative local Euler obstruction to (possibly) non-isolated critical points of functions on spaces which need not be pure-dimensional. 

\medskip

Let $X=\bigcup X_i$ be the decomposition of $X$ into its irreducible components. Let $\Kdot_i$ be a characteristic complex for $X_i$ and let ${\widehat{\mathbf K}}_i^\bullet$ denote the extension by zero of $\Kdot_i$ to all of $X$. Let $f_i$ denote the restriction of $f$ to $X_i$.

\medskip

\begin{defn}\label{def:releuler}  We define the {\bf relative local Euler obstruction of $f$} at $\p\in X$ to be
$$
\operatorname{Eu}_{\mathbf p}f =  \sum_i(-1)^{\dim X_i}\chi\big(\phi_{f_i-f_i(\mathbf p)}[-1]\Kdot_i\big)_{\mathbf p}=  \sum_i(-1)^{\dim X_i}\chi\big(\phi_{f-f(\mathbf p)}[-1]{\widehat{\mathbf K}}_i^\bullet\big)_{\mathbf p}
$$
Note that $\operatorname{Eu}_{\mathbf p}f$ is well-defined by \corref{cor:ccvan}.
\end{defn}

\medskip

\begin{thm}\label{thm:releuler}
The relative local Euler obstruction has the following properties:
\begin{enumerate}
\item If $f\equiv 0$, then $\operatorname{Eu}_{\mathbf p}f=\operatorname{Eu}_{\mathbf p}X$.
\medskip
\item If $(\p, d_\p\tilde f)\not\in\overline{T^*_{X_{\operatorname{reg}}}\U}$, then $\operatorname{Eu}_{\mathbf p}f =0$.
\medskip
\item If $(\p, d_\p\tilde f)$ is an isolated point in $\overline{T^*_{X_{\operatorname{reg}}}\U}\cap\operatorname{im}d\tilde f$, then 
$$
\operatorname{Eu}_{\mathbf p}f= \sum_i(-1)^{\dim X_i}\left(\overline{T^*_{(X_i)_{\operatorname{reg}}}\U}\cdot \operatorname{im}(d\tilde f)\right)_{(\mathbf p, d_{\mathbf p}\tilde f)}
$$
\medskip
\item $\operatorname{Eu}_{\mathbf p}f =\sum_i\operatorname{Eu}_{\mathbf p}f_i$.
\medskip
\item Suppose that $R$ is a principal ideal domain. Let $\mathbf q\in Y$ and suppose that we have a complex analytic function $g:Y\rightarrow\C$. Let $f\boxplus g$ denote the function from $X\times Y$ to $\C$ given by $(f\boxplus g)(x,y)=f(x)+g(y)$. Then $\operatorname{Eu}_{\p\times\mathbf q}(f\boxplus g)=\operatorname{Eu}_{\mathbf p}f\cdot \operatorname{Eu}_{\mathbf q}g$.
\medskip
\end{enumerate}
\end{thm}
\begin{proof} Item 1 follows at once from \corref{cor:eulerob} and \defref{def:releuler}.

\smallskip

Item 4 is immediate from the definition. Item 3 follows immediately from \propref{prop:releuler}. Alternatively, both Items 2 and 3 are immediate from the vanishing cycle index theorem of Ginsburg \cite{ginsburg},  
L\^e \cite{leconcept}, and Sabbah \cite{sabbahquel}, which tells us that, for every bounded, constructible complex $\Adot$ on $X$, if $(\p, d_\p\tilde f)$ is an isolated point in $|\cc(\Adot)|\cap \operatorname{im}(d\tilde f)$, then
$$
\chi\left(\phi_{f-f(\p)}[-1]\Adot\right)_\p = \big(\cc(\Adot)\cdot \operatorname{im}(d\tilde f)\big)_{(\p, d_\p\tilde f)}.
$$

\smallskip

\noindent Item 5:

\smallskip

Let us assume, without loss of generality, that $f(\p)=0$ and $g(\mathbf q)=0$. Let $\Kdot_X$ and $\Kdot_Y$ be characteristic complexes for $X$ and $Y$, Then, we know from \propref{prop:kextprod} that $\Kdot_X\lboxtimes\Kdot_Y$ is a characteristic complex for $X\times Y$.

Then, the derived category version of the Sebastiani-Thom Theorem which we proved in \cite{masseysebthom} tells us that
$$
H^k\big(\phi_{f\boxplus g}(\Kdot_X\lboxtimes\Kdot_Y)\big)_{(\p, \mathbf q)}\cong H^k\big(\phi_f\Kdot_X\lboxtimes\phi_g\Kdot_Y\big)_{(\p, \mathbf q)}\cong
$$
$$
\bigoplus_{i+j=k}H^i(\phi_f\Kdot_X)_\p\otimes H^j(\phi_g\Kdot_Y)_{\mathbf q}\oplus \bigoplus_{i+j=k+1}\operatorname{Tor}\big(H^i(\phi_f\Kdot_X)_\p, H^j(\phi_g\Kdot_Y)_{\mathbf q}\big).
$$
\medskip

\noindent Item 5 follows.
\end{proof}

\medskip

\begin{rem} We naturally refer to Property 5 above as the {\it Sebastiani-Thom property} of the relative local Euler obstruction.
\end{rem}
 
\section{Basics of Enriched Cycles}\label{sec:enrcycles}

In \defref{def:gecc}, we defined the graded, enriched characteristic cycle. In this section, we wish to describe graded, enriched cycles more generally and carefully. We also describe the associated intersection theory. The intersection theory that we use is a fairly simple extension of the intersection theory of properly intersecting cycles in an analytic manifold, as described in section 8.2 of \cite{fulton}. In this case, one obtains intersection {\bf cycles}, not merely rational equivalence classes of cycles.

\vskip .2in

\begin{defn}\label{def:basicenrich} An {\it enriched cycle}, $E$, in $X$ is a formal, locally finite sum $\sum_V E_V[V]$, where the
$V$'s are irreducible analytic subsets of $X$ and the
$E_V$'s are finitely-generated $R$-modules. We refer to the $V$'s as the {\it components\/} of $E$, and to $E_V$ as the {\it
$V$-component module of
$E$}. Two enriched cycles are considered the same provided that all of the component modules are isomorphic. The underlying set
of $E$ is
$|E|:=
\cup_{{}_{E_V\neq 0}}V$.

If $C = \sum n_V[V]$ is an ordinary positive cycle in $X$, i.e., all of the
$n_v$ are non-negative integers, then there is a corresponding enriched cycle $[C]^{\operatorname{enr}}$ in which the
$V$-component module is the free
$R$-module of rank $n_V$. If $R$ is an integral domain, so that rank of an $R$-module is well-defined, then an enriched cycle
$E$ yields an ordinary cycle $[E]^{\operatorname{ord}}:=\sum_V (\operatorname{rk}(E_V))[V]$.

If $q$ is a finitely-generated module and $E$ is an enriched cycle, then we let
$qE:=\sum_V(q\otimes E_V)[V]$; thus, if $R$ is an integral domain and $E$ is an enriched cycle,
$[qE]^{\operatorname{ord}}=(\operatorname{rk}(q))[E]^{\operatorname{ord}}$ and if
$C$ is an ordinary positive cycle and $n$ is a positive integer, then $[nC]^{\operatorname{enr}}=R^n[C]^{\operatorname{enr}}$.

\vskip .1in

The (direct) sum of two enriched cycles $D$ and $E$ is given by $(D + E)_V := D_V\oplus E_V$. 

\vskip .1in

There is a partial ordering on isomorphism classes of finitely-generated $R$-modules given by $M\leq Q$ if and only if there exists a finitely-generated $R$-module $N$ such that $M\oplus N\cong Q$. This relation is clearly reflexive and transitive; moreover, anti-symmetry follows from the fact that if $M$ and $N$ are Noetherian modules such that $M\oplus N\cong M$, then $N=0$. This partial ordering extends to a partial ordering on enriched cycles given by: $D\leq E$ if and only if there exists an enriched cycle $P$ such that $D+P=E$. If the base ring $R$ is a PID and $D+P=E$, then $D$ is uniquely determined by $P$ and $E$, and we write $D=E-P$.

\vskip .1in

If two irreducible analytic subsets $V$ and
$W$ intersect properly in $\U$, then the (ordinary) intersection cycle $[V]\cdot[W]$ is a well-defined positive cycle; we
define the enriched intersection product of $[V]^{\operatorname{enr}}$ and $[W]^{\operatorname{enr}}$ by
$[V]^{\operatorname{enr}}\odot[W]^{\operatorname{enr}} = ([V]\cdot[W])^{\operatorname{enr}}$. If $D$ and $E$ are enriched
cycles, and every component of $D$ properly intersects every component of $E$ in $\U$, then we say that {\it $D$ and $E$
intersect properly\/} in $\U$ and we extend the intersection product linearly, i.e., if $D=\sum_V D_V[V]$ and $E=\sum_W
E_W[W]$, then
$$ D\odot E:= \sum_{V, W} (D_V\otimes E_W)([V]\cdot [W])^ {\operatorname{enr}}.
$$

A {\it graded, enriched cycle $E^\bullet$\/} is simply an enriched cycle $E^i$ for $i$ in some bounded set of integers. An single
enriched cycle is considered as a graded enriched cycle by being placed totally in degree zero. The analytic set $V$ is a {\it
component of $E^\bullet$} if and only if $V$ is a component of $E^i$ for some $i$, and the underlying set of $E^\bullet$ is
$|E^\bullet|=\cup_i|E^i|$. If $R$ is a domain, then $E^\bullet$ yields an ordinary cycle
$[E^\bullet]^{\operatorname{ord}}:=\sum_i(-1)^i(\operatorname{rk}(E^i_V))[V]$. If $k$ is an integer, we define the {\it
$k$-shifted graded, enriched cycle}
$E^\bullet[k]$ by $(E^\bullet[k])^i:= E^{i+k}$.

If $q$ is a finitely-generated module and $E^\bullet$ is a graded enriched cycle, then we define the graded enriched cycle
$qE^\bullet$ by
$(qE^\bullet)^i:=\sum_V(q\otimes E^i_V)[V]$. The (direct) sum of two graded enriched cycles $D^\bullet$ and $E^\bullet$ is given
by $(D^\bullet+ E^\bullet)^i_V := D^i_V\oplus E^i_V$. If $D^i$ properly intersects $E^j$ for all $i$ and $j$, then we say that
$D^\bullet$ and $E^\bullet$ {\it intersect properly\/} and we define the intersection product by 
$$ (D^\bullet\odot E^\bullet)^k:=\sum_{i+j = k}(D^i\odot E^j).
$$

Whenever we use the enriched intersection product symbol, we mean that we are considering the objects on both sides of $\odot$ as graded, enriched cycles, even if we do not superscript by $\operatorname{enr}$ or  $\bullet$.
\end{defn}

\bigskip

Let $\tau: W\rightarrow Y$ be a proper morphism between analytic spaces. If $C = \sum n_V[V]$ is an ordinary positive cycle in
$W$, then the proper push-forward $\tau_*(C) = \sum n_V\tau_*([V])$ is a well-defined ordinary cycle. 

\begin{defn}\label{def:properpush} If $E^\bullet=\sum_V
E^\bullet_V[V]$ is an enriched cycle in $W$, then we define the {\bf proper push-forward of $E^\bullet$ by $\tau$} to be the
graded enriched cycle $\tau^\bullet_*(E^\bullet)$ defined by
$$
\tau^j_*(E^\bullet) \ := \ \sum_V E^j_V[\tau_*([V])]^{\operatorname{enr}}.
$$ 
\end{defn}

The ordinary projection formula for divisors ([{\bf F}], 2.3.c) immediately implies the following enriched version. 

\begin{prop}\label{prop:properpush} Let
$E^\bullet$ be a graded enriched cycle in $X$. Let $W:=|E^\bullet|$. Let $\tau:W\rightarrow Y$ be a proper morphism, and let
$g: Y\rightarrow\mathbb C$ be an analytic function such that
$g\circ\tau$ is not identically zero on any component of $E^\bullet$. Then, $g$ is not identically zero on any component of
$\tau^\bullet_*(E^\bullet)$ and
$$
\tau^\bullet_*\big(E^\bullet\ \odot \ V(g\circ\tau)\big) \ = \ \tau^\bullet_*(E^\bullet) \ \odot \ V(g).
$$
\end{prop}

\section{The Relative Polar Curve}\label{sec:relpolarcurve}

We will use the notation established in \secref{sec:gecc}:  $\U$ is an open neighborhood of the origin in $\C^{n+1}$, $X$ is a closed, analytic subset of $\U$, $\strat$ is a complex analytic Whitney stratification of $X$, with connected strata, $R$ is a base ring (with some technical assumptions), $\Fdot$ is a bounded complex of sheaves of $R$-modules on $X$, which is constructible with respect to $\strat$, $\mathbf z=(z_0, \dots, z_n)$ is a set of coordinates on $\U$, we identify the cotangent space $T^*\U$ with $\U\times\C^{n+1}$ by mapping $(\p, w_0d_{\p}z_0+\dots+w_nd_{\p}z_n)$ to $(\p, (w_0,\dots, w_n))$, for each $S\in\strat$, $d_S=\dim S$, and $(\N_S, \cL_S)$ is the complex Morse data for $S$ in $X$, consisting of a normal slice and complex link of $S$ in $X$.

Let $\tilde f$ and $\tilde g$ be analytic functions from $(\U, \0)$ to $(\C, 0)$, and let $f$ and $g$ denote the restrictions of $\tilde f$ and $\tilde g$, respectively, to $X$.  By refining $\strat$, if necessary, we assume that $V(f)$ is a union of strata. 

\begin{defn}
Let $\strat(\Fdot):=\{S\in\strat\ |\ \hyp^*(\mathbb N_S,\mathbb L_S; \Fdot)\neq 0\}$; we refer to the elements of $\strat(\Fdot)$ as the {\it $\Fdot$-visible strata of $\strat$}.
\end{defn}

Fix a point $\p\in \U$. In  \cite{hammlezariski}, \cite{teissiercargese}, \cite{leattach}, \cite{letopuse}, Hamm, Teissier, and L\^e define and use the relative polar curve (of $\tilde f$ with respect to $z_0$), $\Gamma^1_{\tilde f, z_0}$, to prove a number of topological results related to the Milnor fiber of hypersurface singularities. We shall recall some definitions and results here. We should mention that there are a number of different characterizations of the relative polar curve, all of which agree when $z_0$ is sufficiently generic; below, we have selected what we consider the easiest way of describing the relative polar curve as a set, a scheme, and a cycle.

\smallskip

As a set, $\Gamma^1_{\tilde f, z_0}$ is the closure of the critical locus of $(\tilde f, z_0)$ minus the critical locus of $\tilde f$, i.e., $\Gamma^1_{\tilde f, z_0}$ equals $\overline{\Sigma(\tilde f, z_0)-\Sigma \tilde f}$, as a set. If $z_0$ is sufficiently generic for $\tilde f$ at $p$, then, in a neighborhood of $p$, $\Gamma^1_{\tilde f, z_0}$ will be purely one-dimensional (which includes the possibility of being empty).

It is not difficult to give $\Gamma^1_{\tilde f, z_0}$ a scheme structure. If $\Gamma^1_{\tilde f, z_0}$ is purely one-dimensional at $\p$, then, at points $\mathbf x$ near, but unequal to, $\p$,  $\Gamma^1_{\tilde f, z_0}$ is given the structure of the scheme $\displaystyle V\left(\frac{\partial\tilde f}{\partial z_1}, \dots, \frac{\partial\tilde f}{\partial z_n}\right)$. One can also ``algebraically'' remove any embedded components of $\Gamma^1_{\tilde f, z_0}$ at $\p$ by using {\it gap sheaves\/}; see Chapter 1 of \cite{lecycles}.

In practice, all topological applications of the relative polar curve use only its structure as an analytic cycle (germ), that is, as a locally finite sum of irreducible analytic sets (or germs of sets) counted with integral multiplicities (which will all be non-negative). 

 If $C$ is a one-dimensional irreducible germ of $\Gamma^1_{\tilde f, z_0}$ at $\p$, and $\mathbf x\in C$ is close to, but unequal to, $\p$, then  the component $C$ appears in the cycle $\Gamma^1_{\tilde f, z_0}$ with multiplicity given by the Milnor number of $\tilde f_{|_H}$ at $\mathbf x$, where $H$ is a generic affine hyperplane passing through $\mathbf x$.

\bigskip

Suppose that $M$ is a complex submanifold of $\U$. Recall:

\begin{defn}\label{def:main} The relative conormal space $T^*_{\tilde f_{|_M}}\U$ is given by 
$$
T^*_{\tilde f_{|_M}}\U :=\{(x, \eta)\in T^*\U\ |\ \eta(T_xM\cap \ker d_x\tilde f)=0\}.
$$

If $M\subseteq X$, then $T^*_{\tilde f_{|_M}}\U$ depends on $f$, but not on the particular extension $\tilde f$. In this case, we write $T^*_{f_{|_M}}\U$ in place of $T^*_{\tilde f_{|_M}}\U$.
\end{defn}

\begin{defn}\label{def:relconorm} The {\bf graded, enriched relative conormal cycle, $\big(T^*_{{}_{f,\Fdot}}\U\big)^\bullet$, of $f$, with respect to $\Fdot$}, is defined by
$$\big(T^*_{{}_{f,\Fdot}}\U\big)^k:=\sum_{\substack{S\in\strat(\Fdot)\\ f_{|_S}\neq{\rm\ const.}}}m^k_S(\Fdot)\left[\overline{T^*_{f_{|_S}}\U}\right].$$
\end{defn}

\bigskip

We now wish to define the graded, enriched relative polar curve.  
We will consider the image, $\im d\tilde g$, of $d\tilde g$ in $T^*\U$; this scheme is defined by 
$$V\left(w_0- \frac{\partial \tilde g}{\partial z_0}, \dots, w_n- \frac{\partial \tilde g}{\partial z_n}\right)\subseteq \U\times\C^{n+1}.$$
We will consider $\im d\tilde g$ as a scheme, an analytic set, an ordinary cycle, and as a graded, enriched cycle; we will denote all of these by simply $\im d\tilde g$, and explicitly state what structure we are using or let the context make the structure clear.

Note that the projection $\pi$ induces an isomorphism from the analytic set $\im d\tilde g$ to $\U$. We will use the proper push-forward (\defref{def:properpush}) of the map $\pi$ restricted to $\im d\tilde g$; we will continue to denote this restriction by simply $\pi$. 

  By our conventions in \secref{sec:enrcycles}, the graded, enriched cycle $\im d\tilde g$ is zero outside of degree $0$, and is the enriched cycle $R[\im d\tilde g]$ in degree $0$.

\begin{defn}\label{def:polarcurve} If $S\in\strat$ and $f_{|_S}$ is not constant, we define the  {\bf relative polar set}, $\big|\Gamma_{f,\tilde g}(S)\big|$, to be $\pi\left(\overline{T^*_{f_{|_S}}\U}\ \cap\ \im d\tilde g\right)$; if this set is purely $1$-dimensional, so that $\overline{T^*_{f_{|_S}}\U}$ and $\im d\tilde g$ intersect properly, we define the (ordinary) {\bf relative polar cycle}, $\Gamma_{f,\tilde g}(S)$, to be the cycle $\pi_*\left(\left[\overline{T^*_{f_{|_S}}\U}\right]\cdot [\im d\tilde g]\right)$.

The {\bf relative polar set}, $\big|\Gamma_{f, \tilde g}(\Fdot)\big|$, is defined by 
$$
\big|\Gamma_{f, \tilde g}(\Fdot)\big|:= \pi\left(\big|\big(T^*_{{}_{f,\Fdot}}\U\big)^\bullet\big|\cap \im d\tilde g\right).
$$

Each $1$-dimensional component $C$ of $\big|\Gamma_{f, \tilde g}(\Fdot)\big|$ is the image of a component of $\big|\big(T^*_{{}_{f,\Fdot}}\U\big)^\bullet\big|\cap \im d\tilde g$ along which $\big|\big(T^*_{{}_{f,\Fdot}}\U\big)^\bullet\big|$ and  $\im d\tilde g$ intersect properly. We give such a component $C$ the structure of the graded, enriched cycle whose underlying set is $C$ and whose graded, enriched cycle structure is given by $\pi_*^\bullet\left(\big(T^*_{{}_{f,\Fdot}}\U\big)^\bullet\odot \im d\tilde g\right)$ over generic points in $C$. We refer to this as the {\bf graded, enriched cycle structure of $C$ in $\big|\Gamma_{f, \tilde g}(\Fdot)\big|$}.

If  $\big|\Gamma_{f, \tilde g}(\Fdot)\big|$ is purely $1$-dimensional, we say that the {\bf graded, enriched relative polar curve}, $\big(\Gamma^1_{f, \tilde g}(\Fdot)\big)^\bullet$, is defined, and is given by 
$$
\big(\Gamma^1_{f, \tilde g}(\Fdot)\big)^\bullet:= \pi_*^\bullet\left(\big(T^*_{{}_{f,\Fdot}}\U\big)^\bullet\odot \im d\tilde g\right),
$$
i.e., 
$$
\big(\Gamma^1_{f, \tilde g}(\Fdot)\big)^k = \sum_{\substack{S\in\strat(\Fdot)\\ f_{|_S}\neq{\rm\ const.}}}m^k_S(\Fdot)\left(\Gamma_{f, \tilde g}(S)\right)^{\operatorname{enr}}.
$$
\end{defn}

\begin{rem}\label{rem:polar} If $\tilde g=z_0$ is a generic linear form and $S=\U$, then $\Gamma_{f,\tilde g}(S)$ is the classical polar curve $\Gamma^1_{\tilde f, z_0}$ (as a cycle) of Hamm, L\^e, and Teissier. 

In the notation for the polar curve, we write $\tilde g$, not simply $g$; we do not, in fact, know if $\big(\Gamma^1_{f, \tilde g}(\Fdot)\big)^\bullet$ is independent of the extension to $\tilde g$. However, when $\big(\Gamma^1_{f, \tilde g}(\Fdot)\big)^\bullet$ is defined and has no component on which $f$ is constant, then $\big(\Gamma^1_{f, \tilde g}(\Fdot)\big)^\bullet$ is independent of the extension $\tilde g$. It is also not difficult to show that the set $\big|\Gamma_{f, \tilde g}(\Fdot)\big|$ is independent of the extension of $g$, but we shall not need this result here.

Note that $\overline{T^*_{f_{|_S}}\U}\cap \im d\tilde g$ is at least $1$-dimensional at each point of intersection, and so $\big|\Gamma_{f, \tilde g}(\Fdot)\big|$ has no isolated points. Also, note that, as $\big|\big(T^*_{{}_{f,\Fdot}}\U\big)^\bullet\big|\cap \im d\tilde g$ is a closed subset of $\im d\tilde g$, and $\pi$ induces an isomorphism from $\im d\tilde g$ to $\U$, $\big|\Gamma_{f, \tilde g}(\Fdot)\big|$ is a closed subset of $\U$.

Finally, if the relative polar set is $1$-dimensional, we frequently superscript with a $1$ to emphasize that fact.
\end{rem}

\medskip

For the purposes of this paper, we need to recall the following proposition, which is Proposition 3.13 of \cite{enrichpolar}.

\begin{prop}\label{prop:geng} 

\ 

\begin{enumerate}
\item There exists a non-zero linear form $\tilde\call$ such that $\0\not\in\big|\Gamma_{f, \tilde\call}(\Fdot)\big|$ if and only if for generic linear $\tilde\call$, $\0\not\in\big|\Gamma_{f, \tilde\call}(\Fdot)\big|$.

\item For generic linear $\tilde\call$, 
$$\dim_\0 V(f)\cap \big|\Gamma_{f, \tilde\call}(\Fdot)\big|\leq 0 \hskip 0.2in \textnormal{and}\hskip 0.2in \dim_\0 V(\tilde\call)\cap \big|\Gamma_{f,\tilde\call}(\Fdot)\big|\leq 0.$$
\end{enumerate}
\end{prop}

\medskip

\begin{exm}\label{exm:hardergecc} In this example, we will calculate another graded, enriched characteristic cycle. We shall use this as a basis for the next two examples, in which we calculate a graded, enriched relative conormal cycle and a graded, enriched relative polar cycle.

Let $f:\C^3\rightarrow\C$ be given by $f(x, y, t) =y(y^2-x^3-t^2x^2)$, and let 
$$X:=V(f)=V(y)\cup V(y^2-x^3-t^2x^2).$$
 The singular set of $X$, $\Sigma X$, is the $1$-dimensional set $V(x, y)\cup V(x+t^2, y)$.  Thus, near the origin (actually, in this specific example, globally), 

\smallskip

$\strat:=\{V(y)- V(y^2-x^3-t^2x^2), V(y^2-x^3-t^2x^2)-V(y), \hfill$

$\hfill V(x, y)-\{\0\},  V(x+t^2, y)-\{\0\}, \{\0\}\}$

\smallskip

\noindent is a Whitney stratification of $X$ with connected strata. Let $\Fdot:=\Z^\bullet_X[2]$.  We want to calculate the graded, enriched characteristic cycle of $\Fdot$.
 
\smallskip

First, consider the $2$-dimensional strata. Let $S_1:=V(y)- V(y^2-x^3-t^2x^2)$. Then, $\N_{S_1}$ is simply a point, and $\cL_{S_1}$ is empty. Hence, 
$$H^{k-2}(\N_{S_1}, \cL_{S_1}; \Fdot)=H^{k}(\N_{S_1}, \cL_{S_1}; \Z)$$
 isomorphic to $\Z$ if $k=0$, and is $0$ if $k\neq 0$. The same conclusion holds if $S_1$ is replaced by $S_2:=V(y^2-x^3-t^2x^2)-V(y)$.

Now, consider the $1$-dimensional strata. Let $S_3:= V(x, y)-\{\0\}$, and $S_4:=V(x+t^2, y)-\{\0\}$. The normal slice $\N_{S_3}$ is, as a germ, up to analytic isomorphism, three complex lines in $\C^2$, which intersect at a point, and $\cL_{S_3}$ is three points. Similarly, the normal slice $\N_{S_4}$ is, as a germ, up to analytic isomorphism, two complex lines in $\C^2$, which intersect at a point, and $\cL_{S_4}$ is two points. Hence, $H^{k-1}(\N_{S_3}, \cL_{S_3}; \Fdot)=H^{k+1}(\N_{S_3}, \cL_{S_3}; \Z)$ isomorphic to $\Z^2$ if $k=0$, and is $0$ if $k\neq 0$. Similarly, $H^{k-1}(\N_{S_4}, \cL_{S_4}; \Fdot)=H^{k+1}(\N_{S_4}, \cL_{S_4}; \Z)$ isomorphic to $\Z$ if $k=0$, and is $0$ if $k\neq 0$.

Finally, consider the stratum $\{\0\}$. Then, $\N_{\{\0\}}$ is all of $X$, intersected with a small ball around the origin. The complex link $\cL_{\{\0\}}$ is usually referred to as simply the complex link of $X$ at $\0$. Thus, $\cL_{\{\0\}}$ has the homotopy-type of a bouquet of $1$-spheres (see \cite{levan}), and the number of spheres in this bouquet is equal to the intersection number $(\Gamma^1_{f, L}\cdot V(L))_\0$, where $L$ is any linear form such that $d_\0L$ is not a degenerate covector from strata of $X$ at $0$ (see \cite{stratmorse}), and the relative polar curve here is the classical one of L\^e, Hamm, and Teissier. We claim that we may use $L:=t$ for this calculation.

To see this, first note that $V(y^2-x^3-t^2x^2)$ is the classic example of a space such that the regular part satisfies Whitney's condition (a) along the $t$-axis (or, alternatively, this is an easy exercise). Thus, $d_\0t$ is not a limit of conormals from $S_2$. Now, the closures of $S_1$, $S_3$, and $S_4$ are all smooth, and $d_\0t$ is not conormal to these closures at the origin.

To find the ordinary cycle $\Gamma^1_{f, t}$, we take the components of the cycle below which are {\bf not} contained in $\Sigma f$:
$$
V\left(\frac{\partial f}{\partial x}, \frac{\partial f}{\partial y}\right) =$$
$$ V(y(-3x^2-2t^2x), 3y^2-x^3-t^2x^2)=V(y, x^2(x+t^2)) + V(x(3x+2t^2), 3y^2-x^3-t^2x^2)=
$$
$$
2V(x, y)+V(x+t^2, y) +2V(x, y) + V(3x+2t^2, 3y^2-x^3-t^2x^2).
$$

\smallskip

\noindent Thus, $\Gamma^1_{f, t} = V(3x+2t^2, 3y^2-x^3-t^2x^2)$, and 
$$(\Gamma^1_{f, t}\cdot V(t))_\0 = [V(3x+2t^2, 3y^2-x^3-t^2x^2, t)]_\0 = 2,$$
and $H^{k-0}(\N_{\{\0\}}, \cL_{\{\0\}}; \Fdot)=H^{k+2}(\N_{\{\0\}}, \cL_{\{\0\}}; \Z)$ is isomorphic to $\Z^2$ if $k=0$, and is $0$ if $k\neq 0$.

\smallskip

Therefore, we find that $\gecc^k(\Fdot)= 0$ if $k\neq 0$, and 
$$
\gecc^0(\Fdot) = \Z\left[\overline{T^*_{S_1}\C^3}\right]+ \Z\left[\overline{T^*_{S_2}\C^3}\right] +\Z^2\left[\overline{T^*_{S_3}\C^3}\right]+ \Z\left[\overline{T^*_{S_4}\C^3}\right]+\Z^2 \left[T^*_{\{\0\}}\C^3\right].
$$
\end{exm}

\begin{exm}\label{exm:gercc} We continue with the setting of \exref{exm:hardergecc}, where $$X=V(y)\cup V(y^2-x^3-t^2x^2)$$
 and $\Fdot=\Z^\bullet_X[2]$. We had Whitney strata consisting of $\{\0\}$, 
 $$S_1=V(y)- V(y^2-x^3-t^2x^2), \  \ S_2=V(y^2-x^3-t^2x^2)-V(y),$$
 $$S_3= V(x, y)-\{\0\}\ \textnormal{ and } \ S_4=V(x+t^2, y)-\{\0\}.$$

We found that
$\gecc^k(\Fdot)= 0$ if $k\neq 0$, and 
$$
\gecc^0(\Fdot) = \Z\left[\overline{T^*_{S_1}\C^3}\right]+ \Z\left[\overline{T^*_{S_2}\C^3}\right] +\Z^2\left[\overline{T^*_{S_3}\C^3}\right]+ \Z\left[\overline{T^*_{S_4}\C^3}\right]+\Z^2 \left[T^*_{\{\0\}}\C^3\right].
$$
We will calculate $\big(T^*_{{}_{x,\Fdot}}\C^3\big)^\bullet$.

\smallskip

As we said above, we identify $T^*\C^3$ with $\C^3\times\C^3$, and will use coordinates $(w_0, w_1, w_2)$ for cotangent coordinates, so that $(w_0, w_1, w_2)$ represents $w_0dx+w_1dy+w_2dt$.

\medskip

Since $x$ is identically zero on $\{\0\}$ and $S_3$, these two strata are not used in the calculation of $\big(T^*_{{}_{x,\Fdot}}\C^3\big)^\bullet$. For the $1$-dimensional stratum $S_4$, $\left[\overline{T^*_{x_{|_{S_4}}}\C^3}\right]$ is the $4$-dimensional cycle $V(x+t^2,y)\subseteq \C^3\times\C^3$.

\medskip

The fiber of $T^*_{x_{|_{S_1}}}\C^3$ over any $p\in S_1$ is  
$$(T^*_{S_1}\C^3)_p+<d_px>:=\{\omega+ad_px\ |\ \omega\in (T^*_{S_1}\C^3)_p, a\in\C\}=\{bd_py+ad_px\ |\ a,b\in\C\}.$$ 
Hence, $\left[\overline{T^*_{x_{|_{S_1}}}\C^3}\right]= V(y, w_2)$.

\medskip

The fiber of $T^*_{x_{|_{S_2}}}\C^3$ over any $p\in S_2$ which is a regular point of $x$ restricted to $S_2$ is
$$(T^*_{S_2}\C^3)_p+<d_px>=\{\omega+ad_px\ |\ \omega\in (T^*_{S_2}\C^3)_p, a\in\C\}=$$
$$\{b\big((-3x^2-2t^2x)d_px+2yd_py-2tx^2d_pt\big)+ad_px\ |\ a,b\in\C\}.$$
The form $w_0d_px+w_1d_py+w_2d_pt$ is in this set if and only if the determinant of the following matrix is $0$:
$$\left[\begin{matrix}w_0 & w_1 & w_2\\ -3x^2-2t^2x & 2y & -2tx^2\\ 1 & 0 & 0
\end{matrix}\right],$$
i.e., if and only if $yw_2+tx^2w_1=0$.
It is tempting to conclude that 
$$\left[\overline{T^*_{x_{|_{S_2}}}\C^3}\right] \ = \ V(y^2-x^3-t^2x^2, yw_2+tx^2w_1),$$ 
but this is not the case; we must eliminate any components of 
$$V(y^2-x^3-t^2x^2, yw_2+tx^2w_1)$$
 which are contained in $V(y)$. Obviously, $V(y^2-x^3-t^2x^2, yw_2+tx^2w_1)$ is purely $4$-dimensional, and one easily shows that any component contained in $V(y)$ must, in fact, equal $V(x,y)$ (on the level of sets). Thus, we need to remove any components of $V(y^2-x^3-t^2x^2, yw_2+tx^2w_1)$  which are contained in $V(x,y)$.

Our notation for the resulting scheme (a gap sheaf, see \cite{numcontrol}, I.1) is 
$$V(y^2-x^3-t^2x^2, yw_2+tx^2w_1)\lnot V(x,y).$$
Note that, as schemes,
$$V(y^2-x^3-t^2x^2, yw_2+tx^2w_1)= V(y^2-x^3-t^2x^2, yw_2+tx^2w_1, y^2w_2+ytx^2w_1)=$$
$$
V(y^2-x^3-t^2x^2, yw_2+tx^2w_1, (x^3+t^2x^2)w_2+ytx^2w_1).
$$
Note the $x^2$ factor of the last polynomial listed above, and note that, in the analytic set above, if $x=0$, then $y$ must be $0$, i.e., if a point has $x=0$, the point must be in $V(x,y)$.
Hence, using \cite{numcontrol}, I.1.3.iv, we find that, as cycles, 
$$V(y^2-x^3-t^2x^2, yw_2+tx^2w_1)\lnot V(x,y) =$$
$$ V(y^2-x^3-t^2x^2, yw_2+tx^2w_1, (x+t^2)w_2+ytw_1).$$
(This last equality need not be true on the level of schemes, since our generators do not form a regular sequence and, hence, there may be embedded subvarieties.)

Therefore, we find that $\big(T^*_{{}_{x,\Fdot}}\C^3\big)^k$ is $0$ unless $k=0$, and 

\noindent $\big(T^*_{{}_{x,\Fdot}}\C^3\big)^0=\hfill$
$$\Z[V(y, w_2)]\ +\ \Z[V(y^2-x^3-t^2x^2, yw_2+tx^2w_1, (x+t^2)w_2+ytw_1)]\ +\ \Z[V(x+t^2, y)].
$$
\end{exm}

\medskip

\begin{exm}\label{ex:gerpc} We continue with our setting from \exref{exm:hardergecc} and \exref{exm:gercc}, and consider $X=V(y)\cup V(y^2-x^3-t^2x^2)$ and $\Fdot=\Z^\bullet_X[2]$. We will calculate $\big(\Gamma^1_{x, t}(\Fdot)\big)^\bullet$.

Using the isomorphism $T^*\C^3\cong \C^3\times\C^3$ from \exref{exm:gercc}, $\im dt$ is the scheme 
$$V\left(w_0-\frac{\partial t}{\partial x}, \ w_1-\frac{\partial t}{\partial y}, \ w_2-\frac{\partial t}{\partial t} \right)=V(w_0, w_1, w_2-1).$$

In \exref{exm:gercc}, we found that $\big(T^*_{{}_{x,\Fdot}}\C^3\big)^k$ is $0$ unless $k=0$, and 

\medskip

$\big(T^*_{{}_{x,\Fdot}}\C^3\big)^0=\Z[V(y, w_2)]+\hfill$

$\hfill\Z[V(y^2-x^3-t^2x^2, yw_2+tx^2w_1, (x+t^2)w_2+ytw_1)]\ +\ \Z[V(x+t^2, y)].$

\medskip

Let 
$$E \ = \ V(y^2-x^3-t^2x^2, yw_2+tx^2w_1, (x+t^2)w_2+ytw_1),$$ 
on the level of cycles, throughout the remainder of this example.

Thus, $\big(\Gamma^1_{x, t}(\Fdot)\big)^k$ is $0$ unless $k=0$ and, to calculate $\big(\Gamma^1_{x, t}(\Fdot)\big)^0$, we need first to calculate the three ordinary cycles
$$\pi_*\big(V(y, w_2)\cdot V(w_0, w_1, w_2-1)\big),$$
$$\pi_*\big(E\cdot V(w_0, w_1, w_2-1)\big),$$
and
$$\pi_*\big(V(x+t^2, y)\cdot V(w_0, w_1, w_2-1)\big).$$

Now, $V(y, w_2)\cap V(w_0, w_1, w_2-1)=\emptyset$, and so $\pi_*\big(V(y, w_2)\cdot V(w_0, w_1, w_2-1)\big)=0$. In addition, it is trivial that there is an equality of cycles 
$$\pi_*\big(V(x+t^2, y)\cdot V(w_0, w_1, w_2-1)\big)= V(x+t^2, y).$$ However, the remaining cycle is more difficult to calculate.

\smallskip

It is trivial that, as sets, 
$$E\cap V(w_0, w_1, w_2-1)= V(x+t^2, y, w_0, w_1, w_2-1), 
$$
but the difficulty in calculating 
$$\pi_*\big(E\cdot V(w_0, w_1, w_2-1)\big)$$
is due to the fact that $y^2-x^3-t^2x^2, yw_2+tx^2w_1, (x+t^2)w_2+ytw_1$ is not a regular sequence. To ``fix'' this, note that, in \exref{exm:gercc}, we saw that, as cycles, there is an equality
$$
V(y^2-x^3-t^2x^2, yw_2+tx^2w_1)= C+E,
$$
where the underlying set $|C|=V(x, y)$. Therefore,
$$
C\cdot V(w_0, w_1, w_2-1) + E\cdot V(w_0, w_1, w_2-1) = $$
$$V(y^2-x^3-t^2x^2, yw_2+tx^2w_1)\cdot V(w_0, w_1, w_2-1) = 
$$
$$
V(y^2-x^3-t^2x^2, yw_2+tx^2w_1, w_0, w_1, w_2-1)= V(x^2(x+t^2), y, w_0, w_1, w_2-1)=
$$
$$
2V(x,y,w_0, w_1, w_2-1)+ V(x+t^2, y, w_0, w_1, w_2-1).
$$
Thus, as cycles,
$$
E\cdot V(w_0, w_1, w_2-1) = V(x+t^2, y, w_0, w_1, w_2-1),
$$
and so $\pi_*(E\cdot V(w_0, w_1, w_2-1))= V(x+t^2, y)$.

Finally, we find that
$$
\big(\Gamma^1_{x, t}(\Fdot)\big)^0= \pi_*^0\left((T^*_{{}_{f,\Fdot}}\U\big)^\bullet\odot \im dt\right) =$$ $$
\Z[V(x+t^2, y)]+  \Z[V(x+t^2, y)] =  \Z^2[V(x+t^2, y)].
$$
\end{exm}

\bigskip

We shall discuss the main results on the graded, enriched relative conormal cycle and the graded, enriched relative polar curve in the following sections.

\section{The Nearby Cycles}\label{sec:nearby}

The following theorem was our primary motivation for defining the graded, enriched conormal cycle. While we state the theorem in the elegant form given in \cite{singenrich}, this theorem, in terms of ordinary cycles,  is essentially contained in \cite{bmm}

\bigskip

\begin{thm}\label{thm:psigecc} {\rm (\cite{singenrich}, Theorem 3.3)} There is an equality of graded enriched cycles given by
$$
\gecc^\bullet(\psi_f[-1]\Fdot) \ = \ \big(T^*_{{}_{f, \Fdot}}\U\big)^\bullet\odot(V(f)\times\C^{n+1}).
$$ 
\end{thm}

\medskip

If one knows the irreducible components $\{V_j\}_j$ of the underlying set $\ms(\psi_f[-1]\Fdot)$, then, by selecting a generic point of each $V_j$, and taking a normal slice, the calculation of $\gecc^\bullet(\psi_f[-1]\Fdot)$  is reduced to calculating the Morse modules of point strata. In other words, if we know $\ms(\psi_f[-1]\Fdot)$, then, by taking normal slices, the calculation of $\gecc^\bullet(\psi_f[-1]\Fdot)$ reduces to calculating $m^k_\0(\psi_f[-1]\Fdot)$, i.e., $H^k(\phi_{\call}[-1]\psi_f[-1]\Fdot)_\0$, where $\call$ is the restriction to $V(f)$ of a generic linear form $\tilde\call$.

\bigskip

The next result follows from Theorem 4.2 of \cite{hypercohom}, but is stated as in Theorem 3.12 of \cite{enrichpolar}.

\begin{thm}\label{thm:vanpsi}
$$
m^k_\0(\psi_f[-1]\Fdot)\cong \left(\big(\Gamma^1_{f, \tilde\call}(\Fdot)\big)^k\odot V(f)\right)_\0 =\bigoplus_{\substack{S\in\strat(\Fdot)\\ S\not\subseteq V(f)}}m^k_S(\Fdot)\otimes R^{\alpha_S},
$$
where $\alpha_S:= \left(\Gamma^1_{f,\tilde\call}(S)\cdot V(\tilde f)\right)_\0$
and $\tilde\call$ is a generic linear form. Specifically, the amount of genericity that we need is that $(\0, d_\0\tilde \call)\not\in \overline{\ms(\psi_f[-1]\Fdot)-T^*_\0\U}$, which is equivalent to $\dim_\0 |\big(\Gamma^1_{f, \tilde\call}(\Fdot)\big)^\bullet|\cap V(f)\leq 0$.
\end{thm}

\bigskip

\begin{exm}\label{exm:phigecc} We continue with our example from \exref{exm:gercc}: 
$$X=V(y)\cup V(y^2-x^3-t^2x^2),$$ 
$\Fdot=\Z^\bullet_X[2]$, and we found that
$\big(T^*_{{}_{x,\Fdot}}\C^3\big)^k$ is $0$ unless $k=0$, and 

\smallskip

\noindent $\big(T^*_{{}_{x,\Fdot}}\C^3\big)^0=\Z[V(y, w_2)] +\hfill$

\smallskip

$\hfill\ \Z[V(y^2-x^3-t^2x^2, yw_2+tx^2w_1, (x+t^2)w_2+ytw_1)]\ +\ \Z[V(x+t^2, y)].$

\medskip

\noindent Let $\tilde f:= x$.

\bigskip

In light of \thmref{thm:psigecc}, we find that $\gecc^\bullet(\psi_f[-1]\Fdot)$ is concentrated in degree $0$, and that

\medskip

$\gecc^0(\psi_f[-1]\Fdot) =\Z[V(x, y, w_2)] +\Z^2[V(x, y, t)] +\hfill  {(\dagger)}$

\medskip

$\hfill \Z\big[V(y^2-x^3-t^2x^2, yw_2+tx^2w_1, (x+t^2)w_2+ytw_1)\cdot V(x)\big]^{\operatorname{enr}}.$

\medskip

\noindent The difficulty is in calculating the cycle $$E:=\left[V(y^2-x^3-t^2x^2, yw_2+tx^2w_1, (x+t^2)w_2+ytw_1)\cdot V(x)\right].$$

The underlying set $|E|$ is easily found to be $V(x, y, t)\cup V(x, y, w_2)$, and we may find the geometric multiplicity of each component in $E$ by moving to generic points. 

At a generic point of $V(x, y, t)$, $w_2\neq 0$ and, at such a point, one easily shows that there is an equality of ideals
$$
\langle y^2-x^3-t^2x^2, yw_2+tx^2w_1, (x+t^2)w_2+ytw_1\rangle = \langle yw_2+tx^2w_1, (x+t^2)w_2+ytw_1\rangle
$$
and, as $yw_2+tx^2w_1, (x+t^2)w_2+ytw_1$ is a regular sequence, one easily calculates that, at a point where $w_2\neq 0$, there are equalities of cycles
$$
\left[V(yw_2+tx^2w_1, (x+t^2)w_2+ytw_1)\cdot V(x)\right] = $$
$$
\left[V(yw_2+tx^2w_1, (x+t^2)w_2+ytw_1, x)\right] = [V(y, t^2, x)] = 2[V(x, y, t)].
$$

\smallskip

\noindent This is the component of $E$ with underlying set $V(x, y, t)$.

At a generic point of $V(x, y, w_2)$, neither $t$ nor $w_1$ is zero, and it follows that $x+t^2$ is not zero. At such a point, one easily shows that there is again an equality of ideals
$$
\langle y^2-x^3-t^2x^2, yw_2+tx^2w_1, (x+t^2)w_2+ytw_1\rangle = \langle yw_2+tx^2w_1, (x+t^2)w_2+ytw_1\rangle
$$
and one easily calculates that, at a point where neither $t$ nor $w_1$ is zero, there are equalities of cycles
$$
\left[V(yw_2+tx^2w_1, (x+t^2)w_2+ytw_1)\cdot V(x)\right] = \left[V(yw_2, tw_2+yw_1, x)\right] = 
$$
$$
[V(y, tw_2, x)]+[V(w_2, yw_1, x)]=[V(y, w_2, x)]+[V(w_2, y, x)] = 2[V(x, y, w_2)].
$$
This is the component of $E$ with underlying set $V(x, y, w_2)$.

\medskip

Therefore, $(\dagger)$ tells us that 
$$
\gecc^0(\psi_f[-1]\Fdot)=
\Z^3[V(x, y, w_2)]\ +\ \Z^4[V(x, y, t)].
$$

\medskip

Note that the component of $\gecc^0(\psi_f[-1]\Fdot)$ over the origin agrees with \thmref{thm:vanpsi} and our calculation in \exref{ex:gerpc}. For 
$$(\0, d_\0 t)\not\in \overline{\ms(\psi_f[-1]\Fdot)-T^*_\0\U}= V(x, y, w_2)
$$
and 
$$
\big(\Gamma^1_{x, t}(\Fdot)\big)^0\odot V(t)  =  \Z^2[V(x+t^2, y)]\odot V(t) = \Z^4[V(x, y, t)].
$$
\end{exm}

\bigskip

Another example, which we leave as an exercise for the reader, is to recalculate $\gecc^\bullet(\psi_f[-1]\Fdot)$ from \exref{exm:curvegecc} by using either \thmref{thm:psigecc} or \thmref{thm:vanpsi}.

\bigskip

Recall from \defref{def:numerical} that a complex of sheaves $\Pdot$ on $X$ is numerical if and only if $\Pdot$ is a perverse sheaf and all of its Morse modules are free, so that the ordinary characteristic cycle of $\Pdot$ carries all of the information about $\gecc^\bullet(\Pdot)$.

\medskip

The following corollary is immediate from \thmref{thm:psigecc}.
\begin{cor}\label{cor:numnearnum} If $\Pdot$ is numerical, then so is $\psi_f[-1]\Pdot$.
\end{cor}

\section{Hypersurface Complements and Restrictions}\label{sec:complrestr}

Let $i:X-V(f)\hookrightarrow X$  and $j:V(f)\hookrightarrow X$ denote the inclusions. Recall that we are assuming that $V(f)$ is a union of strata, and recall the partial ordering on isomorphism classes of finitely-generated $R$-modules given in \defref{def:basicenrich}.

\medskip

We would like to give an elegant formula for $\gecc^\bullet(i_!i^!\Fdot)$, something along the lines of what we gave for $\gecc^\bullet(\psi_f[-1]\Fdot)$ in \thmref{thm:psigecc}. We do not quite do this. However, we do the next best thing; we give a formula for the set $\ms(i_!i^!\Fdot)$ and a formula for $m^k_\0(i_!i^!\Fdot)$. 

Once we have these formulas, and so, in principle, know $\gecc^\bullet(i_!i^!\Fdot)$, we can use how the graded, enriched characteristic cycle works with Verdier duals to obtain $\gecc^\bullet(i_*i^*\Fdot)$. In addition, we can use the additivity of ordinary characteristic cycles over distinguished triangles, and the duality formula, to obtain the characteristic cycles of $j_*j^*[-1]\Fdot$ and $j_!j^![1]\Fdot$ when the base ring is a domain.

\medskip

The following result is immediate from Theorem 4.2 B of \cite{hypercohom} and, for ordinary cycles, is proved in \cite{bmm}.

\begin{thm}\label{thm:vanrestrext}
$$
m^k_\0(i_!i^!\Fdot)\cong \left(\big(\Gamma^1_{f, \tilde\call}(\Fdot)\big)^k\odot V(\tilde\call)\right)_\0=\bigoplus_{\substack{S\in\strat(\Fdot)\\ S\not\subseteq V(f)}}m^k_S(\Fdot)\otimes R^{\beta_S},
$$
where $\beta_S:= \left(\Gamma^1_{f,\tilde\call}(S)\cdot V(\tilde\call)\right)_\0$
and $\tilde\call$ is a generic linear form. Specifically, the amount of genericity that we need is that 

\begin{enumerate}
\item $(\0, d_\0\tilde \call)\not\in \overline{\ms(i_!i^!\Fdot)-T^*_\0\U}$ and 
\item $\dim_\0 |\big(\Gamma^1_{f, \tilde\call}(\Fdot)\big)^\bullet|\cap V(\call)\leq 0$.
\end{enumerate}
\end{thm}

\bigskip

\begin{cor}\label{cor:morserestrext} If $S\in\strat$ and $S\not\subseteq V(f)$, then $m^k_S(i_!i^!\Fdot)\cong m^k_S(\Fdot)$.

If $S\in\strat$ and $S\subseteq V(f)$, then $m^k_S(i_!i^!\Fdot)\neq 0$ if and only if $m^k_S(\psi_f[-1]\Fdot)\neq 0$, and $m^k_S(i_!i^!\Fdot)\leq m^k_S(\psi_f[-1]\Fdot)$.
\end{cor}
\begin{proof} Outside of $V(f)$, the complex $i_!i^!\Fdot$ agrees with $\Fdot$; this yields the first statement.

A comparison of \thmref{thm:vanrestrext} with \thmref{thm:vanpsi} yields the second statement.
\end{proof}

\bigskip

\begin{defn} Suppose that $E^\bullet$ is a graded, enriched cycle in $T^*\U$ given by
$E^k=\sum_{S\in\strat}E^k_S[\overline{T^*_S\U}]$, where $E^k_S$ is a finitely-generated $R$-module.

Let $(E_{\not\subseteq V(f)})^\bullet$ be the graded, enriched cycle such that $(E_{\not\subseteq V(f)})^k$ is the sum of those $E^k_S[\overline{T^*_S\U}]$ such that $S\not\subseteq V(f)$. Similarly, let $(E_{\subseteq V(f)})^\bullet$ be the graded, enriched cycle such that $(E_{\subseteq V(f)})^k$ is the sum of those $E^k_S[\overline{T^*_S\U}]$ such that $S\subseteq V(f)$.

Let $|E^k|_{\not\subseteq V(f)}:= |(E_{\not\subseteq V(f)})^\bullet|$ and $|E^k|_{\subseteq V(f)}:= |(E_{\subseteq V(f)})^\bullet|$.
\end{defn}

\bigskip

\begin{cor}\label{cor:restrextgeccset} 
$$
|\gecc^k(i_!i^!\Fdot)|= |\gecc^k(\Fdot)|_{\not\subseteq V(f)}\cup |\gecc^k(\psi_f[-1]\Fdot)|.
$$
\end{cor}
\begin{proof} This is immediate from \corref{cor:morserestrext}.
\end{proof}

\bigskip

In light of \corref{cor:restrextgeccset} and the hypotheses on $\tilde\call$ in \thmref{thm:vanrestrext}, the following proposition is of interest.

\begin{prop} Suppose that $(\0, d_\0\tilde\call)\not\in|\ms(\Fdot)|_{\not\subseteq V(f)}$. 

Then, $\dim_\0 |\big(\Gamma^1_{f, \tilde\call}(\Fdot)\big)^\bullet|\cap V(\call)\leq 0$ if and only if $\dim_\0 |\big(\Gamma^1_{f, \tilde\call}(\Fdot)\big)^\bullet|\cap V(f)\leq 0$.
\end{prop}
\begin{proof} Lemma 3.10 of \cite{enrichpolar} tells us that, if $\dim_\0 |\big(\Gamma^1_{f, \tilde\call}(\Fdot)\big)^\bullet|\cap V(\call)\leq 0$, then $\dim_\0 |\big(\Gamma^1_{f, \tilde\call}(\Fdot)\big)^\bullet|\cap V(f)\leq 0$. 

Suppose then that $\dim_\0 |\big(\Gamma^1_{f, \tilde\call}(\Fdot)\big)^\bullet|\cap V(f)\leq 0$. Let $\mathbf p(t)$ be an analytic parametrization of an irreducible component $C$ of $|\big(\Gamma^1_{f, \tilde\call}(\Fdot)\big)^\bullet|$ such that $\mathbf p(0)=\0$. Suppose that $C\subseteq V(\tilde\call)$; we wish to derive a contradiction.

Let $S^\prime\in\strat$ be an $\Fdot$-visible stratum such that $C=\pi(\overline{T^*_{f_{|_{S^\prime}}}\U}\cap\im d\tilde\call)$. Let $S$ denote the stratum of $\strat$ which contains $\mathbf p(t)$ for $t\neq 0$. Note that neither $S$ nor $S^\prime$ is contained in $V(f)$, since $\dim_\0 |\big(\Gamma^1_{f, \tilde\call}(\Fdot)\big)^\bullet|\cap V(f)\leq 0$. On the other hand, in a neighborhood of the origin, the stratified critical locus of $f$ is contained in $V(f)$.

 It follows that, for all $\mathbf x\in C-\{\0\}$, the fiber $(\overline{T^*_{f_{|_{S^\prime}}}\U})_{\mathbf x}$ is equal to $(\overline{T^*_{S^\prime}\U})_{\mathbf x}+\langle d_{\mathbf x}\tilde f\rangle$. Thus, for $t\neq 0$, there exists a complex number $a(t)$ such that 
$$
d_{\mathbf p(t)}\tilde\call+a(t)d_{\mathbf p(t)}\tilde f\in (\overline{T^*_{S^\prime}\U})_{\mathbf p(t)}\subseteq (T^*_{S}\U)_{\mathbf p(t)}.\eqno{(\dagger)}
$$
By evaluating at $\mathbf p^\prime(t)$, we find that $(\call(\mathbf p(t)))^\prime+a(t)(f(\mathbf p(t))^\prime\equiv 0$. As $C\subseteq V(\tilde\call)$, we find that $a(t)(f(\mathbf p(t))^\prime\equiv 0$. As $C\not\subseteq V(f)$, we conclude that $a(t)\equiv 0$. From $(\dagger)$, it follows that $d_{\mathbf p(t)}\tilde\call\in (\overline{T^*_{S^\prime}\U})_{\mathbf p(t)}$ and, hence, that $(\0, d_\0\tilde\call)\in \overline{T^*_{S^\prime}\U}$. This contradicts the fact that $(\0, d_\0\tilde\call)\not\in|\ms(\Fdot)|_{\not\subseteq V(f)}$.
\end{proof}

\bigskip

\begin{thm}\label{thm:vanrestrpush} For all $k$, $\gecc^k(i_*i^*\Fdot)=\gecc^k(i_!i^!\Fdot)$.
\end{thm}
\begin{proof} One uses \thmref{thm:vanrestrext} and \corref{cor:restrextgeccset}, together with the isomorphisms $i_*i^*\Fdot\cong\vdual i_!i^!\vdual\Fdot$ and $\vdual\psi_f[-1]\cong\psi_f[-1]\vdual$. We leave the proof as an exercise.
\end{proof}

\medskip

\begin{rem} Note that $\gecc^k(i_*i^*\Fdot)$ and $\gecc^k(i_!i^!\Fdot)$ do not depend on any degree of $\gecc^\bullet(\Fdot)$, other than the degree $k$ portion. In particular, if $\gecc^\bullet(\Fdot)$ is concentrated in degree $0$, then so are $\gecc^\bullet(i_*i^*\Fdot)$ and $\gecc^\bullet(i_!i^!\Fdot)$. Thus, we recover the well-known fact that, if $\Fdot$ is a perverse sheaf and $i$ is the inclusion of the complement of the zero locus of a single function, then $i_*i^*\Fdot$ and $i_!i^!\Fdot$ are also perverse.
\end{rem}

\medskip

The following corollary is immediate from \thmref{thm:vanrestrext} and \thmref{thm:vanrestrpush}.
\begin{cor}\label{cor:numcompnum} If $\Pdot$ is numerical, then so are $i_*i^*\Pdot$ and $i_!i^!\Pdot$.
\end{cor}

\smallskip

\noindent\rule{1in}{1pt}

\bigskip

Recall that we have the closed inclusion $j:V(f)\hookrightarrow X$. Unlike the functors $i_*i^*$ and $i_!i^!$ discussed above, the functors $j_*j^*$ and $j_!j^!$ (with or without shifts) do not take perverse sheaves to perverse sheaves. In other words, $\gecc^k(j_*j^*\Fdot)$ and $\gecc^k(j_!j^!\Fdot)$ are not determined by a single degree of $\gecc^\bullet(\Fdot)$. Of course, given the canonical distinguished triangles relating $j_*j^*$ and  $i_!i^!$, and $j_!j^!$ and $i_*i^*$, we immediately obtain:

\begin{prop}\label{prop:easytri} 

\ 

$$\cc(j_*j^*\Fdot) = \cc(\Fdot)-\cc(i_!i^!\Fdot) =  \cc(\Fdot)-\cc(i_*i^*\Fdot) = \cc(j_!j^!\Fdot).$$
\end{prop}

\bigskip

Perhaps more interesting is the easy corollary to \thmref{thm:vanrestrext}:

\medskip

\begin{cor}\label{cor:hyplink} Suppose that, for all $k$, $m^k_\0(\Fdot)=0$. Then, 
$$
m^k_\0(j_*j^*[-1]\Fdot)\cong m^k_\0(j_!j^![1]) \cong \left(\big(\Gamma^1_{f, \tilde\call}(\Fdot)\big)^k\odot V(\tilde\call)\right)_\0=\bigoplus_{\substack{S\in\strat(\Fdot)\\ S\not\subseteq V(f)}}m^k_S(\Fdot)\otimes R^{\beta_S},
$$
where $\beta_S:= \left(\Gamma^1_{f,\tilde\call}(S)\cdot V(\tilde\call)\right)_\0$
and $\tilde\call$ is a generic linear form.

\end{cor}
\begin{proof} Let $\call =\tilde\call_{|_{X}}$. The corollary follows immediately from \thmref{thm:vanrestrext} and \thmref{thm:vanrestrpush}, together with applying the functor $\phi_\call[-1]$ to the two canonical distinguished triangles
$$
\cdots\rightarrow j_*j^*[-1]\Fdot\rightarrow i_!i^!\Fdot\rightarrow \Fdot\arrow{[1]}\cdots
$$
and
$$
\cdots\rightarrow\Fdot\rightarrow i_*i^*\Fdot\rightarrow  j_!j^![1]\Fdot\arrow{[1]}\cdots.
$$
\end{proof}

\bigskip

We wrote that the above corollary is ``more interesting'' because it lets us conclude the classical result below.

\begin{exm}\label{exm:genlink} Suppose that we have $f:(\U,\0)\rightarrow(\C,0)$, where $\U$ is connected and $f$ is not identically $0$. Consider the constant sheaf $\Pdot:=\Z^\bullet_{{}_\U}[n+1]$. 

Then, for all $k$, $m^k_\0(\Pdot)=0$ and so, by \corref{cor:hyplink},
$$m^k_\0(j_*j^*[-1]\Pdot)\cong\bigoplus_{\substack{S\in\strat(\Pdot)\\ S\not\subseteq V(f)}}m^k_S(\Pdot)\otimes R^{\beta_S}$$
where $\beta_S:= \left(\Gamma^1_{f,\tilde\call}(S)\cdot V(\tilde\call)\right)_\0$
and $\tilde\call$ is a generic linear form.

\smallskip

The only $\Pdot$-visible stratum is $\U$. As $\Pdot$ is perverse, the only possibly non-zero $m^k_\U(\Pdot)$ occurs when $k=0$, and $m^0_\U(\Pdot)\cong \Z$.

\smallskip

Therefore, the only possibly non-zero $m^k_\0(j_*j^*[-1]\Pdot)$ occurs when $k=0$ and
$$
m^0_\0(j_*j^*[-1]\Pdot)\cong \Z^\beta,
$$
where $\beta = \left(\Gamma^1_{f,\tilde\call}(\U)\cdot V(\tilde\call)\right)_\0$.

\smallskip

This says that
$$
m^0_\0(j_*j^*[-1]\Pdot)=m^0_\0(\Z^\bullet_{{}_{V(f)}}[n]) = \left(\Gamma^1_{f,L}\cdot V(L)\right)_\0
$$
for generic linear $L$, where, since we are in affine space, we have written the more usual $L$ in place of $\tilde\call$ and have written $\Gamma^1_{f,L}$ in place of $\Gamma^1_{f,\tilde\call}(\U)$. Note that $m^0_\0(\Z^\bullet_{{}_{V(f)}}[n])$ is precisely the reduced integral cohomology in degree $n-1$ of the complex link of $V(f)$ at $\0$. 

\smallskip

This is the cohomological version of the well-known result that the complex link of an $n$-dimensional hypersurface has the homotopy-type of a bouquet of $(n-1)$-spheres, where the number of spheres is given by $\left(\Gamma^1_{f,L}\cdot V(L)\right)_\0$.
\end{exm}

\section{The Vanishing Cycles}\label{sec:vanishing}

Before we can give a formula for $\gecc^\bullet(\phi_f[-1]\Fdot)$, we must first discuss the graded, enriched exceptional divisor in the blow-up of a graded, enriched cycle along an ideal.

Let us recall the notation established thus far. $\U$ is an open neighborhood of the origin of $\C^{n+1}$,  $X$ is a closed, analytic subset of $\U$, $\mathbf z:=(z_0, \dots, z_n)$ are coordinates on $\U$, we identify the cotangent space $T^*\U$ with $\U\times\C^{n+1}$ by mapping $(\p, w_0d_{\p}z_0+\dots+w_nd_{\p}z_n)$ to $(\p, (w_0,\dots, w_n))$., and we let $\pi:T^*\U\rightarrow\U$ denote the projection.  

Consider a graded, enriched cycle $D^\bullet$ in $T^*\U$ given by $D^k:=\sum D^k_V[V]$. Let $h_0, \dots, h_m$ be analytic functions on $T^*\U$, and let $I$ be the ideal $\langle h_0, \dots h_m\rangle$. Then, for each $V$, the blow-up ${\operatorname{Bl}}_I V$ of $V$ along $I$ is naturally a subspace of $T^*\U\times\Proj^m\cong \U\times \C^{n+1}\times\Proj^m$. Let ${\operatorname{Ex}}_I(V)$ denote the exceptional divisor as a cycle. 

\begin{defn} The {\bf graded, enriched blow-up ${\operatorname{Bl}}^\bullet_I(D^\bullet)$ of $D^\bullet$ along $I$ in $T^*\U\times\Proj^m$} is given by ${\operatorname{Bl}}^k_I(D^\bullet):=\sum_V D^k_V[{\operatorname{Bl}}_I V]^{\operatorname{enr}}$.

The {\bf graded, enriched exceptional divisor ${\operatorname{Ex}}^\bullet_I(D^\bullet)$ of $D^\bullet$ along $I$ in $T^*\U\times\Proj^m$} is given by ${\operatorname{Ex}}^k_I(D^\bullet):=\sum_V D^k_V[{\operatorname{Ex}}_I V]^{\operatorname{enr}}$.
\end{defn}

\smallskip

Instead of subscripting the blow-up and exceptional divisor by the ideal $I$, it is common to subscript by the analytic scheme $V(I)$. In particular, below, we shall blow-up along $\operatorname{im}d\tilde f\subseteq \U\times\C^{n+1}$; we remind the reader that this is defined by the ideal
$$\left\langle w_0- \frac{\partial \tilde f}{\partial z_0}, \dots, w_n- \frac{\partial \tilde f}{\partial z_n}\right\rangle.$$

\bigskip

Let $\tau:\U\times\C^{n+1}\times\Proj^n\rightarrow \U\times\Proj^n$ denote the projection, and recall that $\tau_*$ denotes the proper push-forward. The following is Theorem 3.5 of \cite{singenrich}.

\vskip .2in

\begin{thm}\label{thm:vangecc}
There is an equality of closed subsets of $X$ given by
$$\bigcup_{v\in\C}\supp\phi_{f-v}[-1]\Fdot\ =\ \pi\big(\ms(\Fdot)\cap\operatorname{im}d\tilde f\big),
$$ 
and, for all $k$, an equality of graded, enriched cycles given by
$$
\sum_{v\in\C}\Proj\big(\gecc^k(\phi_{f-v}[-1]\Fdot)\big)\ =\ \tau_*\big({\operatorname{Ex}}_{\operatorname{im}d\tilde f}(\gecc^k(\Fdot))\big).
$$ 
In particular, for all $k$, there is an equality of sets
$$\bigcup_{v\in\C}\pi\big(|\gecc^k(\phi_{f-v}[-1]\Fdot)|\big)\ =\ \pi\big(|\gecc^k(\Fdot)|\cap\operatorname{im}d\tilde f\big).
$$
\end{thm}

\bigskip

\begin{rem}
We should remark that, in the above unions and sum over $v\in\C$, the unions and sum are not merely locally finite, but, in fact, locally over open neighborhoods of points in $X$, there is only one non-zero (or non-empty) summand (respectively, indexed subset in the union).
\end{rem}

\medskip

The following corollary is immediate from \thmref{thm:vangecc}.
\begin{cor}\label{cor:vannum} If $\Pdot$ is numerical, then so is $\phi_f[-1]\Pdot$.
\end{cor}

\bigskip

The following result follows at once from Theorem 4.2 of \cite{hypercohom}, and is used in the proof of \thmref{thm:vangecc}.

\begin{thm}\label{thm:vanphi}
$$
m^k_\0(\phi_f[-1]\Fdot)\cong 
m^k_\0(\Fdot)\oplus\bigoplus_{\substack{S\in\strat(\Fdot)\\ S\not\subseteq V(f)}}m^k_S(\Fdot)\otimes R^{\delta_S},
$$
where $\delta_S:= \left(\Gamma^1_{f,\tilde\call}(S)\cdot V(\tilde f)\right)_\0-\left(\Gamma^1_{f,\tilde\call}(S)\cdot V(\tilde\call)\right)_\0$, where $\tilde\call$ is a generic linear form; specifically, we need for the following three conditions to hold:
\begin{enumerate}
\item $(\0, d_\0\tilde \call)\not\in \overline{\ms(\phi_f[-1]\Fdot)-T^*_\0\U}$;

\item $\dim_\0 |\big(\Gamma^1_{f, \tilde\call}(\Fdot)\big)^\bullet|\cap V(\call)\leq 0$, and 

\item $\left(\Gamma^1_{f,\tilde\call}(S)\cdot V(\tilde f)\right)_\0\geq\left(\Gamma^1_{f,\tilde\call}(S)\cdot V(\tilde\call)\right)_\0$.
\end{enumerate}
\end{thm}

\medskip

Before we can state and prove our next result, we need a definition.

\begin{defn} The {\bf algebraic critical locus of $f$}, $\Sigma_{\operatorname{alg}}f$, is the set of those $x\in X$ such that $f\in \mathfrak m_{{}_{X, x}}^2$, where $\mathfrak m_{{}_{X, x}}$ is the maximal ideal of $X$ at $x$. \end{defn}

\medskip

Now we have:

\begin{thm}\label{thm:sskphi} For all $k$, 
$$\big|\gecc^k(\Fdot)\big|_{\subseteq V(f)}\ \subseteq\ \big|\gecc^k(\phi_f[-1]\Fdot)\big|\ \subseteq\ \big|\gecc^k(\Fdot)\big|_{\subseteq V(f)}\cup \big|\gecc^k(\psi_f[-1]\Fdot)\big|.
$$

Moreover, over $\Sigma_{\operatorname{alg}}f$, this second containment is an equality, i.e.,

\smallskip

\noindent$\big|\gecc^k(\phi_f[-1]\Fdot)\big|\cap\pi^{-1}\left(\Sigma_{\operatorname{alg}}f\right)\ =\ \hfill$

\smallskip

\noindent$\hfill\left(\big|\gecc^k(\Fdot)\big|_{\subseteq V(f)}\cup \big|\gecc^k(\psi_f[-1]\Fdot)\big|\right)\cap\pi^{-1}\left(\Sigma_{\operatorname{alg}}f\right).$
\end{thm}
\begin{proof} The first line of containments follows immediately from \thmref{thm:vanphi} and \thmref{thm:vanpsi}.

The equality over $\Sigma_{\operatorname{alg}}f$ follows at once from \thmref{thm:psigecc}, \thmref{thm:vangecc}, and Theorem 4.2 of \cite{critpts}.
\end{proof}

\vskip .3in

\begin{rem}\label{rem:lecyclerem} In this long remark, we wish to address how effectively one can calculate $\gecc^\bullet(\phi_f[-1]\Fdot)$, given $f$ and $\gecc^\bullet(\Fdot)$. This will require us to discuss much of our work in \cite{numinvar}, \cite{numcontrol}, and \cite{singenrich}. As usual, we will identify $T^*\U$ with $\U\times\C^{n+1}$, and use four different projections: $\tau:\U\times\C^{n+1}\times\Proj^n\rightarrow \U\times\Proj^n$, $\eta:\U\times\Proj^n\rightarrow\U$, $\nu:\U\times\C^{n+1}\times\Proj^n\rightarrow \U\times\C^{n+1}$, and $\pi:\U\times\C^{n+1}\rightarrow\U$.

\smallskip

We shall assume that our base ring $R$ is a PID, and that $X$ has codimension at least $1$ in $\U$ (so that our projectivizations below do not totally discard components).

\smallskip

Assume that we have re-chosen $\U$ small enough so that $\Proj\big(\gecc^k(\phi_f[-1]\Fdot)\big)$ is the only non-zero summand in \thmref{thm:vangecc}. Then, \thmref{thm:vangecc} gives a nice, elegant algebraic characterization of the projectivized $\gecc$ of $\phi_f[-1]\Fdot$, in terms of blow-ups and exceptional divisors. The problem is that blow-ups and exceptional divisors are not so easy to calculate.

Suppose that $\Adot$ is a bounded complex of sheaves of modules over $R$, which is constructible with respect to $\strat$. We shall first describe a general method for ``calculating'' $\gecc^\bullet(\Adot)$, and then describe in the case of $\gecc^\bullet(\phi_f[-1]\Fdot)$ why this really leads to an effective calculation.

First, projectivize and obtain $\Proj(\gecc^k(\Adot))=\sum_{S\in\strat}m^k_S(\Adot)[\Proj(\overline{T^*_S\U})]\subseteq \U\times\Proj^n$. Recall that our coordinates $\mathbf z=(z_0, \dots, z_n)$ determine our cotangent coordinates $(w_0, \dots, w_n)$ and, hence, determine projective coordinates $[w_0:\dots :w_n]$ on $\Proj^n$. We assume that $\U$ is small enough and that our coordinates $\mathbf z$ are generic enough so that, for all $S$ such that $\overline{T^*_S\U}$ is a component of $\ms(\Adot)$, for all $j$ such that $0\leq j\leq n$, the intersection of  $\Proj(\overline{T^*_S\U})$ and $\U\times\Proj^j\times\{\0\}$ in $\U\times\Proj^n$ is proper, and so is purely $j$-dimensional. We claim that the proper push-forwards 
$${}^k\Gamma^j_{{}_{\Adot, \mathbf z}}:=\eta_*\big(\Proj(\gecc^k(\Adot))\odot \U\times\Proj^j\times\{\0\}\big)
$$ 
completely determine $\Proj(\gecc^k(\Adot))$ (and, hence, $\gecc^k(\Adot)$). The ${}^k\Gamma^j_{{}_{\Adot, \mathbf z}}$ are the {\bf characteristic polar cycles}; we refer the reader to Section 5 of \cite{singenrich}.

\medskip

The characteristic polar cycles determine $\Proj(\gecc^k(\Adot))$ by downward induction on the dimension of strata of $X$. Let $d:=\dim X$, which we are assuming is at most $n$.

Consider first a stratum $S$ of dimension $d$. Then, $\overline{T^*_S\U}$ appears in $\gecc^k(\Adot)$ if and only if $\overline{S}$ is a component of $\big | {}^k\Gamma^d_{{}_{\Adot, \mathbf z}}\big |$. In addition, as 
$$\eta_*\Big(m^k_S(\Adot)[\overline{T^*_S\U}]\ \odot\  \U\times\Proj^d\times\{\0\}\Big)\ =\ m^k_S(\Adot)\Big[\eta_*\Big([\overline{T^*_S\U}]\ \odot\  \U\times\Proj^d\times\{\0\}\Big)\Big],
$$
once we know that $\overline{T^*_S\U}$ appears in $\gecc^k(\Adot)$ and we know  ${}^k\Gamma^d_{{}_{\Adot, \mathbf z}}$, then we can determine $m^k_S(\Adot)$. Note that in this process, we do not actually determine the stratum $S$, but rather a closed analytic set which agrees with $S$ on an open dense set -- but this is enough.

Now, suppose that we know the pieces of the enriched cycle $\Proj(\gecc^k(\Adot))$ for all of those strata of dimension at least $j+1$. Let us write $D^k_{\geq j+1}$ for the (enriched) sum of these pieces. Then, one can calculate the enriched cycle $\eta_*\big(D^k_{\geq j+1}\odot \U\times\Proj^j\times\{\0\}\big)$; this cycle is an enriched form of the $j$-dimensional absolute polar varieties of the strata of dimension at least $j+1$. Now, one can consider the difference (we use that $R$ is a PID here)
$$M^j:={}^k\Gamma^j_{{}_{\Adot, \mathbf z}}- \eta_*\big(D^k_{\geq j+1}\odot \U\times\Proj^j\times\{\0\}\big).
$$
Suppose that $S$ is a stratum of dimension $j$. Then, one easily sees that $\overline{T^*_S\U}$ appears in $\gecc^k(\Adot)$ if and only if $\overline{S}$ is a component of $\big |M^j\big |$. As above, once we know that $\overline{T^*_S\U}$ appears in $\gecc^k(\Adot)$ and we know $M^j$, we can determine $m^k_S(\Adot)$ by calculating $\Big[\eta_*\Big([\overline{T^*_S\U}]\ \odot\  \U\times\Proj^j\times\{\0\}\Big)\Big]$.

\bigskip

We have seen above that the characteristic polar cycles determine $\gecc^\bullet(\Adot)$. The question now is: how does one effectively calculate the characteristic polar cycles in the case where $\Adot=\phi_f[-1]\Fdot$?

\bigskip

Let us adopt the notation ${}^k\Lambda^j_{{}_{\Fdot, \mathbf z}} := {}^k\Gamma^j_{{}_{\phi_f[-1]\Fdot, \mathbf z}}$. From our discussion above, we see that we could reconstruct $\gecc^\bullet(\phi_f[-1]\Fdot)$ if we knew the ${}^k\Lambda^j_{{}_{\Fdot, \mathbf z}}$. The result of Corollary 6.8 of \cite{singenrich} gives an algorithm for calculating the ${}^k\Lambda^j_{{}_{\Fdot, \mathbf z}}$, assuming that the coordinates $\mathbf z$ are {\bf $\phi_f[-1]\Fdot$-isolating}. Let us put off the discussion of what $\phi_f[-1]\Fdot$-isolating means; for now, simply assume that the coordinates are generic enough to make true what we write below.

\smallskip

We work in each degree separately; so, fix $k$.

\smallskip

Let $\Pi^{n+1}:= \gecc^k(\Fdot)$. Then, $\Pi^{n+1}$ properly intersects $\displaystyle V\left(w_n-\frac{\partial\tilde f}{\partial z_n}\right)$, and we may consider the enriched cycle defined by the intersection $$\sum_V M_V[V] \ := \ \Pi^{n+1}\odot V\left(w_n-\frac{\partial\tilde f}{\partial z_n}\right);$$ this enriched cycle may have some components contained in $\imdf$ and some components not contained in $\imdf$. Let $\displaystyle\Pi^n:= \sum_{V\not\subseteq\imdf} M_V[V]$ and let $\displaystyle\Delta^n:=\sum_{V\subseteq\imdf}M_V[V]$.

\vskip .1in

Now, proceed inductively: if we have $\Pi^{j+1}$, then $\displaystyle V\left(w_j-\frac{\partial\tilde f}{\partial z_j}\right)$ properly intersects $\Pi^{j+1}$, and we define $\Pi^j$ and $\Delta^j$ by the equality
$$
\Pi^{j+1}\odot V\left(w_j-\frac{\partial\tilde f}{\partial z_j}\right)  \ = \ \Pi^j+\Delta^j,$$
where no component of $\Pi^j$ is contained in $\imdf$, and every component of $\Delta^j$ is contained in $\imdf$.

\vskip .1in

Continue with this process until one obtains $\Pi^0$ and $\Delta^0$.

\vskip .2in

Then, for all $j$, as germs at $\mathbf p$, ${}^k\Lambda^j_{{}_{f, \mathbf z}}(\Fdot) = \pi_*(\Delta^j)$ and  
$$\big[{}^k\Lambda^j_{{}_{f, \mathbf z}}(\Fdot)\odot V(z_0-p_0, \dots, z_{j-1}-p_{j-1})\big]_\mathbf p \cong
$$
$$
H^k(\phi_{z_j-p_j}[-1]\psi_{z_{j-1}-p_{j-1}}[-1]\dots\psi_{z_0-p_0}[-1]\phi_f[-1]\Fdot)_\mathbf p,
$$
where, when $j=0$, we mean
$$
\big[{}^k\Lambda^0_{{}_{f, \mathbf z}}(\Fdot)\big]_\mathbf p\cong H^k(\phi_{z_0-p_0}[-1]\phi_f[-1]\Fdot)_\mathbf p.
$$
Note that, as we are interested in the end only in the $\Delta^j$, throughout the algorithm above, we may, in each step, discard any components of $\Pi^j$ which do not intersect $\imdf$.

\bigskip

The above works very well for calculating the germs of ${}^k\Lambda^j_{{}_{f, \mathbf z}}(\Fdot)$ at $\mathbf p$, and so $\gecc^\bullet(\phi_f[-1]\Fdot)$ above a neighborhood of $\mathbf p$, as long as the coordinates $\mathbf z$ are {\it $\phi_f[-1]\Fdot$-isolating at $\mathbf p$}. In \cite{singenrich}, we give two characterizations of $\phi_f[-1]\Fdot$-isolating that are relevant here. 

Let $s:=\dim_\mathbf p\supp\phi_f[-1]\Fdot=\dim_\mathbf p\pi(\ms(\phi_f[-1]\Fdot))$. Then, the coordinates $\mathbf z$ are $\phi_f[-1]\Fdot$-isolating at $\mathbf p$ if and only if, for all $j$ such that $0\leq j\leq s-1$, $\mathbf p$ is an isolated point in the support of 
$$\phi_{z_j-p_j}[-1]\psi_{z_{j-1}-p_{j-1}}[-1]\dots\psi_{z_0-p_0}[-1]\phi_f[-1]\Fdot.$$ This is equivalent to:  

\smallskip

\noindent for all $j$ such that $0\leq j\leq s-1$, there exists an open neighborhood $\W$ of $\mathbf p$ in $\U$ such that 
$\Proj\big(\ms(\phi_f[-1]\Fdot)\big)$ properly intersects $\W\times\Proj^j\times\{\mathbf 0\}$ inside $\W\times\Proj^n$ and 
\begin{equation}\label{eq:intersectcond}\dim_\mathbf p \Big(V(z_0-p_0, \dots, z_{j-1}-p_{j-1})\cap \eta\Big(\Proj\big(\ms(\phi_f[-1]\Fdot)\big)\ \cap\ \W\times\Proj^j\times\{\mathbf 0\}\Big)\Big)\leq 0.
\end{equation}
When $j=0$, this condition is interpreted as 
$$
\dim_\mathbf p \eta\Big(\Proj\big(\ms(\phi_f[-1]\Fdot)\big)\ \cap\ \W\times\{[1:0:0:\cdots:0]\}\Big)\leq 0.
$$

\medskip

Are either one of these characterizations of $\phi_f[-1]\Fdot$-isolating useful? Yes - the latter one is. \thmref{thm:sskphi} tells us that 
$$\ms(\phi_f[-1]\Fdot)\subseteq \big(\ms(\Fdot)\big)_{\subseteq V(f)}\cup \ms(\psi_f[-1]\Fdot).$$
 So, if our coordinates are generic enough so that \equref{eq:intersectcond} above holds with $\Proj\big(\ms(\phi_f[-1]\Fdot)\big)$ replaced by $\Proj\big(\ms(\Fdot)\big)_{\subseteq V(f)}\cup \Proj\big(\ms(\psi_f[-1]\Fdot)\big)$, then the entire process above works.
\end{rem}

\bigskip

Now we can give an example.

\begin{exm}\label{exm:vangecc} We continue with our earlier situation: 
$$X=V(y)\cup V(y^2-x^3-t^2x^2)\subseteq\C^3, \ \ \Fdot=\Z^\bullet_X[2], \ \textnormal{ and } \ \tilde f:= x.$$ This will give us an easy, but nonetheless, illustrative example of the procedure described in \remref{rem:lecyclerem}.

From \exref{exm:hardergecc}, we know that $\ms(\Fdot)_{\subseteq V(f)} = T^*_{V(x,y)}\U\cup T^*_\0\U$. From \exref{exm:phigecc}, we know that $\ms(\psi_f[-1]\Fdot)$ is also equal to $T^*_{V(x,y)}\U\cup T^*_\0\U$. Therefore, \thmref{thm:sskphi} tells us that
$$
\ms(\phi_f[-1]\Fdot) \ \subseteq \ T^*_{V(x,y)}\U\cup T^*_\0\U,
$$
(in fact, \thmref{thm:sskphi} tells us that this an equality, though we will not use this stronger fact).

Hence,
$$
\supp \phi_f[-1]\Fdot \ = \ \pi\big(\ms(\phi_f[-1]\Fdot)\big) \ \subseteq \ \pi\big(T^*_{V(x,y)}\U\cup T^*_\0\U\big) \ = \ V(x,y).
$$
Considering how simple this set is, we could calculate $\gecc^\bullet(\phi_f[-1]\Fdot)$ ``barehandedly'', by applying \thmref{thm:vanphi} at the origin, and then moving to a generic point on $V(x,y)$, taking a hyperplane slice, and applying \thmref{thm:vanphi} again.

However, we want to demonstrate the procedure that we described in \exref{rem:lecyclerem}. Hence, we will first determine $\phi_f[-1]\Fdot$-isolating coordinates  at $\0$, and then go through the graded, enriched cycle calculation from \exref{rem:lecyclerem}.

\medskip

From the above, we see that $s=\dim_\mathbf 0\supp\phi_f[-1]\Fdot\leq 1$, and thus our coordinates are  $\phi_f[-1]\Fdot$-isolating  at $\0$ if \equref{eq:intersectcond} holds for $\mathbf p=\0$ and $j=0$; this is the degenerate case mentioned immediately after \equref{eq:intersectcond}.

It follows that, if we let $(z_0, z_1, z_2)=(t,x,y)$, so that the cotangent coordinates $(w_0, w_1, w_2)$ correspond to $w_0dt+w_1dx+w_2dy$, then $\Proj^0=\{[1:0:\cdots:0]\}$ in \equref{eq:intersectcond} corresponds to the projective class $[dt]$
and

\medskip

\noindent $\Proj\big(\ms(\phi_f[-1]\Fdot)\big)\ \cap\ \W\times\{[1:0:0:\cdots:0]\} \ \subseteq\hfill$

\smallskip

\noindent$\Big(\Proj\big(T^*_{V(x,y)}\U\big)\cup \Proj\big(T^*_\0\U\big)\Big)\ \cap\ \W\times\{[1:0:0:\cdots:0]\} \ = \emptyset \ \cup  \ \{(\0, [1:0:\cdots:0])\}.$

\bigskip

Therefore,
$$
\dim_\mathbf 0 \eta\Big(\Proj\big(\ms(\phi_f[-1]\Fdot)\big)\ \cap\ \W\times\{[1:0:0:\cdots:0]\}\Big)\leq 0,
$$
and the coordinates $(t,x,y)$ are $\phi_f[-1]\Fdot$-isolating  at $\0$. Note that this ordering on the coordinates is different from what we used earlier, because we need for $t$ to come first.

\bigskip

We can now proceed with the enriched cycle calculation as described in  \exref{rem:lecyclerem}.

\bigskip

As we saw in \exref{exm:hardergecc}, $\gecc^k(\Fdot)= 0$ if $k\neq 0$; thus, we need calculate only in the fixed degree $k=0$.  As we also saw in   \exref{exm:hardergecc},
$$
\gecc^0(\Fdot) = \Z\left[\overline{T^*_{S_1}\C^3}\right]+ \Z\left[\overline{T^*_{S_2}\C^3}\right] +\Z^2\left[\overline{T^*_{S_3}\C^3}\right]+ \Z\left[\overline{T^*_{S_4}\C^3}\right]+\Z^2 \left[T^*_{\{\0\}}\C^3\right],
$$
where 

$$S_1=V(y)- V(y^2-x^3-t^2x^2), \ S_2=V(y^2-x^3-t^2x^2)-V(y), \ S_3= V(x, y)-\{\0\},$$
and $S_4=V(x+t^2, y)-\{\0\}$.

 Using a computer algebra system to find equations defining $\overline{T^*_{S_2}\C^3}$, we have
 
 \medskip
 
$\gecc^0(\Fdot) = \Z\left[V(y, w_0, w_1)\right] +\hfill$

\medskip

$\Z \Big[V\big(y^2-x^3-t^2x^2,  \ 2tw_0w_1-w_0^2-2t^2w_2^2x,  \ -w_0w_1+2tw_1^2-tw_2^2(2t^2+3x), 
$

\medskip

$2w_1x+tw_0+3w_2y, \  w_0^2(t^2+x)-t^2w_2^2x^2, \  w_0y+tw_2x^2, \  2w_1y+w_2(2t^2+3x)x,$

\medskip

$tw_2y+w_0(t^2+x)\big)\Big] +\hfill$

\medskip

$\Z^2\left[V(x,y,w_0)\right]+ \Z\left[V(y, x+t^2, 2tw_1-w_0)\right]+\Z^2 \left[V(t,x,y)\right].\hfill$

\smallskip

Before we proceed with the algorithm, note that 
$$\imdf \ = \ \U\times\{(0,1,0)\} \ = \ V(w_0, w_1-1, w_2).$$

\smallskip

We now let $\Pi^3=\gecc^0(\Fdot)$, and calculate
$$
\Pi^3 \ \odot  \ V\left(w_2-\frac{\partial x}{\partial y}\right) \ = \ \Pi^3 \ \odot  \ V\left(w_2\right) \ =
$$
$$
\Z\left[V(y,  x+t^2,  2tw_1-w_0,  w_2)\right]  \ + \ \Z^2\left[V(x,y,w_0, w_2)\right]+ $$
$$
\Z\left[V(y, x+t^2, 2tw_1-w_0, w_2)\right]+\Z^2 \left[V(t,x,y, w_2)\right]
$$
$$
+ \ \textnormal{components which do not intersect} \ \imdf.
$$
Therefore, we have
$$
\Delta^2 \ = \  \Z^2\left[V(x,y,w_0, w_2)\right]  \ + \ \Z^2 \left[V(t,x,y, w_2)\right],
$$
and
$$
\Pi^2 \ = \ \Z^2\left[V(y,  t^2+x,  2tw_1-w_0,  w_2)\right].
$$

\end{exm}

\section{What about intersection cohomology?}\label{sec:ic}

The characteristic cycle of the intersection cohomology complex is of great importance in representation theory (see \cite{kazhdanlusztigtop} and \cite{bradenpams}), and yet, aside from the curve case in \exref{exm:curvegecc}, we have not discussed the calculation of the characteristic cycle (enriched or not) for intersection cohomology complexes (with constant or twisted coefficients). This is because such a calculation is, not surprisingly, hard, and we have no satisfactory results in this area. 

What may be surprising is that the calculation of the characteristic cycle of intersection cohomology, with constant coefficients, is closely related to the relative Milnor monodromy of the constant sheaf along a hypersurface containing the singular set. We will describe this relationship briefly.

\medskip

Suppose that $X$ is an analytic space and, as we are happy to work locally at $\0$, assume that we have an analytic function $f:X\to\C$ such that the singular set of $X$ is contained in $V(f)$, but that $f$ does not vanish on any irreducible component of $X$.

\smallskip

As before, let $i:X-V(f)\hookrightarrow X$  and $j:V(f)\hookrightarrow X$ denote the inclusions. 

\smallskip

Note that, as intersection cohomology $\Idot$, with constant coefficients, on $X$ is a perverse sheaf, the graded, enriched characteristic cycle is concentrated in degree $0$. Also, the case that is of concern in representation theory is when the base ring is a field. Consequently, we would be satisfied with calculating $\cc(\Idot)$.

Now, \propref{prop:easytri} tells us that we can calculate $\cc(\Idot)$ if we can calculate $\cc(j_*j^*[-1]\Idot)=-\cc(j_*j^*\Idot)$ and $\cc(i_!i^!\Idot)$. But $i^!\Idot\cong i^*\Idot$ is the restriction of $\Idot$ to a generic subset of the smooth part of $X$, by our assumptions on $f$. Thus, $i^!\Idot$ coincides with the restriction of the constant sheaf to a generic smooth subset of $X$. Hence, $\cc(i_!i^!\Idot)$ can be calculated using \thmref{thm:vanrestrext} and its corollary.

We are left with the problem of calculating $\cc(j_*j^*[-1]\Idot)$, which, after slicing, reduces to the problem of calculating the Morse module $m_\0^0(j_*j^*[-1]\Idot)$ or, more precisely, reduces to knowing when this Morse module is zero and when it is not. 

There is the fundamental distinguished triangle, relating the nearby and vanishing cycles:
$$
j_*j^*[-1]\Idot\rightarrow \psi_f[-1]\Idot \arrow{\operatorname{can}} \phi_f[-1]\Idot\arrow{[1]},
$$
which is actually a short exact sequence in the Abelian category of perverse sheaves, due to the fact that  $j_*j^*[-1]\Idot$ is perverse (which uses that $\Idot$ is intersection cohomology). 

There is also the dual variation triangle 
$$
\phi_f[-1]\Idot\arrow{\operatorname{var}}\psi_f[-1]\Idot\rightarrow j_!j^![1]\Idot\arrow{[1]},
$$
which is also a short exact sequence in the Abelian category of perverse sheaves, due to the fact that  $j_!j^![1]\Idot$ is perverse. 

This is where the monodromy automorphism $T_f:\psi_f[-1]\Idot\to \psi_f[-1]\Idot$ comes in. It is well-known that $\operatorname{var}\circ\operatorname{can}= \operatorname{id}-T_f$. It follows that, in the Abelian category of perverse sheaves on $V(f)$, $j_*j^*[-1]\Idot\cong\operatorname{ker}\{\operatorname{id}-T_f\}$.

Suppose now that $\tilde\call$ is a generic linear form, and that $\call$ is the restriction of $\tilde\call$ to $V(f)$. Then, it follows that $\phi_\call[-1]j_*j^*[-1]\Idot$, which is a finite-dimensional vector space concentrated in degree $0$, is isomorphic to the kernel of the map induced by $\operatorname{id}-T_f$ on $\phi_\call[-1]\psi_f[-1]\Idot$, and so is determined by  relative Milnor monodromy.

As $\Idot$ agrees with the constant sheaf on the complement of $V(f)$, which is all that $\psi_f[-1]\Idot$ cares about, we can calculate $m_\0^0(\psi[-1]\Idot)\cong\phi_\call[-1]\psi[-1]\Idot$ via \thmref{thm:vanpsi}, in the easy case of the constant sheaf, where the relevant strata are open dense subsets of the smooth parts of the components of $X$. Moreover, the relative monodromy that we need to analyze is also that from the ``easy'' constant sheaf case.

\medskip

It is, of course, our hope to analyze the above relative monodromy, and produce a method for calculating, in principle and in practice, characteristic cycles of intersection cohomology.

\section{Concluding Remarks}\label{sec:conclude}

It is somewhat annoying in \exref{exm:vangecc}, and in the general algorithm given in \remref{rem:lecyclerem}, that, essentially, we first have to know $\ms(\phi_f[-1]\Fdot)$ in order to begin the calculation of the cycles ${}^k\Lambda^j_{{}_{\Fdot, \mathbf z}}$. 

Why do we have to know $\ms(\phi_f[-1]\Fdot)$ first? Solely because we need to know that our coordinates are $\phi_f[-1]\Fdot$-isolating. Ideally, we could begin with the calculation of the ${}^k\Lambda^j_{{}_{\Fdot, \mathbf z}}$, and check ``on-the-fly'' that certain intersections are proper, which would then tell us that the coordinate choice is generic enough. This is what happens with the L\^e cycles for affine hypersurface singularities; see \cite{lecycles}.

Unfortunately, while we suspect that such a result is true, we have yet to find a proof.

\newpage
\bibliographystyle{plain}
\bibliography{Masseybib}

\begin{thebibliography}{10}

\bibitem{bradenpams}
{Braden, T.}
\newblock {On the Reducibility of Characteristic Varieties}.
\newblock {\em Proc. AMS}, 130:2037--2043, 2002.

\bibitem{bmps}
{Brasselet, J.-P., Massey, D., Parameswaram, Seade, J.}
\newblock {Euler Obstruction and Functions with Isolated Singularities}.
\newblock {\em J. London Math. Soc. (2)}, 70:59--76, 2004.

\bibitem{bmm}
{Brian\c con, J., Maisonobe, P., and Merle, M.}
\newblock { Localisation de syst\`emes diff\'erentiels, stratifications de
  Whitney et condition de Thom}.
\newblock {\em Invent. Math.}, 117:531--550, 1994.

\bibitem{bdk}
{Brylinski, J. L., Dubson, A., and Kashiwara, M.}
\newblock {Formule de l'indice pour les modules holonomes et obstruction
  d'Euler locale}.
\newblock {\em C. R. Acad. Sci., S\'erie A.}, 293:573--576, 1981.

\bibitem{dimcasheaves}
{Dimca, A.}
\newblock {\em {Sheaves in Topology}}.
\newblock Universitext. Springer-Verlag, 2004.

\bibitem{fulton}
{Fulton, W.}
\newblock {\em {Intersection Theory}}, volume~2 of {\em Ergeb. Math.}
\newblock Springer-Verlag, 1984.

\bibitem{ginsburg}
{Ginsburg, V.}
\newblock {Characteristic Varieties and Vanishing Cycles}.
\newblock {\em Invent. Math.}, 84:327--403, 1986.

\bibitem{stratmorse}
{Goresky, M. and MacPherson, R.}
\newblock {\em {Stratified Morse Theory}}, volume~14 of {\em Ergeb. der Math.}
\newblock Springer-Verlag, 1988.

\bibitem{hammlezariski}
{Hamm, H. and L\^e D. T.}
\newblock {Un th\'eor\`eme de Zariski du type de Lefschetz}.
\newblock {\em Ann. Sci. \'Ec. Norm. Sup.}, 6 (series 4):317--366, 1973.

\bibitem{kashsch}
{Kashiwara, M. and Schapira, P.}
\newblock {\em {Sheaves on Manifolds}}, volume 292 of {\em Grund. math.
  Wissen.}
\newblock Springer-Verlag, 1990.

\bibitem{kazhdanlusztigtop}
{Kazhdan, D. and Lusztig, G.}
\newblock {A topological approach to Springer's representations}.
\newblock {\em Adv. Math.}, 38:222--228, 1980.

\bibitem{leattach}
{L\^e, D. T.}
\newblock {Calcul du Nombre de Cycles \'Evanouissants d'une Hypersurface
  Complexe}.
\newblock {\em Ann. Inst. Fourier, Grenoble}, 23:261--270, 1973.

\bibitem{letopuse}
{L\^e, D. T.}
\newblock {Topological Use of Polar Curves}.
\newblock {\em Proc. Symp. Pure Math.}, 29:507--512, 1975.

\bibitem{levan}
{L\^e, D. T.}
\newblock {Sur les cycles \'evanouissants des espaces analytiques}.
\newblock {\em C. R. Acad. Sci. Paris, S\'er. A-B}, 288:A283--A285, 1979.

\bibitem{leconcept}
{L\^e, D. T.}
\newblock {Le concept de singularit\'e isol\'ee de fonction analytique}.
\newblock {\em Advanced Studies in Pure Math.}, 8:215--227, 1986.

\bibitem{maceuler}
{MacPherson, R.}
\newblock {Chern classes for singular varieties}.
\newblock {\em Annals of Math.}, 100:423--432, 1974.

\bibitem{numinvar}
{Massey, D.}
\newblock {Numerical Invariants of Perverse Sheaves}.
\newblock {\em Duke Math. J.}, 73(2):307--370, 1994.

\bibitem{lecycles}
{Massey, D.}
\newblock {\em {L\^e Cycles and Hypersurface Singularities}}, volume 1615 of
  {\em Lecture Notes in Math.}
\newblock Springer-Verlag, 1995.

\bibitem{hypercohom}
{Massey, D.}
\newblock {Hypercohomology of Milnor Fibres}.
\newblock {\em Topology}, 35:969--1003, 1996.

\bibitem{critpts}
{Massey, D.}
\newblock {Critical Points of Functions on Singular Spaces}.
\newblock {\em Top. and Appl.}, 103:55--93, 2000.

\bibitem{masseysebthom}
{Massey, D.}
\newblock {The Sebastiani-Thom Isomorphism in the Derived Category}.
\newblock {\em Compos. Math.}, 125:353--362, 2001.

\bibitem{numcontrol}
{Massey, D.}
\newblock {\em {Numerical Control over Complex Analytic Singularities}}, volume
  778 of {\em Memoirs of the AMS}.
\newblock AMS, 2003.

\bibitem{singenrich}
{Massey, D.}
\newblock {Singularities and Enriched Cycles}.
\newblock {\em Pacific J. Math.}, 215, no. 1:35--84, 2004.

\bibitem{enrichpolar}
{Massey, D.}
\newblock {Enriched Relative Polar Curves and Discriminants}.
\newblock {\em Contemp. Math.}, 474:107--144, 2008.

\bibitem{paruprag}
{Parusi\'nski, A. and Pragacz, P.}
\newblock {Characteristic classes of hypersurfaces and characteristic cycles}.
\newblock {\em J. Alg. Geom.}, 10(1):63--79, 2001.

\bibitem{sabbahquel}
{Sabbah, C.}
\newblock {Quelques remarques sur la g\'eom\'etrie des espaces conormaux}.
\newblock {\em Ast\'erisque}, 130:161--192, 1985.

\bibitem{schurbook}
{Sch\"urmann, J.}
\newblock {\em {Topology of Singular Spaces and Constructible Sheaves}},
  volume~63 of {\em Monografie Matematyczne}.
\newblock {Birkh\"auser}, 2004.

\bibitem{teissiercargese}
{Teissier, B.}
\newblock {Cycles \'evanescents, sections planes et conditions de Whitney}.
\newblock {\em Ast\'erisque}, 7-8:285--362, 1973.

\end{thebibliography}
\end{document}